\documentclass{amsart}
\usepackage{amssymb}
\usepackage{booktabs}
\usepackage{multirow}
\usepackage{mathrsfs}
\usepackage{mathtools}
\usepackage{psfrag}
\usepackage{xr-hyper}
\usepackage[colorlinks]{hyperref}
\makeatletter
\@namedef{subjclassname@2020}{%
  \textup{2020} Mathematics Subject Classification}
\makeatother



\title[Birkhoff sums as distributions II: Deformations ]{Birkhoff sums as distributions II:\\ {\small Applications to deformations of dynamical systems}}

\author[C. Grotta-Ragazzo]{Clodoaldo Grotta-Ragazzo}
\address{ Instituto de Matem\'atica e Estat\'istica,  Universidade de S\~ao Paulo, Rua do Mat\~ao, 1010, Cidade Universit\'aria, S\~ao Paulo-SP, CEP 05508-090, Brazil }
\email{ragazzo@usp.br} 
\author[D.  Smania]{Daniel Smania}
\address{Departamento de Matem\'atica, Instituto de Ci\^encias Matem\'aticas e de Computa\c{c}\~ao (ICMC), Universidade de S\~ao Paulo (USP), Avenida Trabalhador S\~ao-carlense, 400, S\~ao Carlos-SP, CEP 13566-590,  Brazil.}
\email{smania@icmc.usp.br} 
\urladdr{\url{https://sites.icmc.usp.br/smania/}}
\thanks{C.G.R. is partially supported by FAPESP, Brazil grant 2016/25053-8. D.S. was partially supported by CNPq 306622/2019-0, CNPq 307617/2016-5, CNPq Universal 430351/2018-6 and FAPESP Projeto Tem\'atico 2017/06463-3.}

\newtheorem{theorem}{Theorem}[section]
\newtheorem{corollary}[theorem]{Corollary}

\newtheorem{lemma}[theorem]{Lemma}
\newtheorem{proposition}[theorem]{Proposition}

\theoremstyle{definition}

\newtheorem{remark}[theorem]{Remark}

\renewcommand{\theequation}{\thesection.\arabic{equation}}
\newcommand\numberthis{\addtocounter{equation}{1}\tag{\theequation}} 

\usepackage{constants} 
\newcommand{\secdot}[1]{\arabic{#1}}
\newconstantfamily{c}{
symbol=\lambda,
format=\secdot,
}
\renewconstantfamily{normal}{
symbol=C,
format=\secdot,
}

\newconstantfamily{e}{
symbol=\theta,
format=\arabic,
}

\newconstantfamily{A}{
symbol=A,
format=\arabic,
}

\newconstantfamily{G}{
symbol=G,
format=\arabic,
}

\newcommand{\Cll}[2][normal]{\Cl[#1]{#2}}
\newcommand{\Crr}[1]{\Cr{#1}}
\newcounter{change}

\usepackage[colorinlistoftodos, textwidth=4cm]{todonotes}

\makeatletter
\providecommand\@dotsep{5}
\renewcommand{\listoftodos}[1][\@todonotes@todolistname]{%
  \@starttoc{tdo}{#1}}
\makeatother

\hypersetup{ colorlinks=true, pdftitle={Birkhoff sums as distributions: applications to deformations of dynamical systems}, pdfsubject={ distribution, Birkhoff sums, deformation, topological classes, piecewise expanding maps}, pdfauthor={Clodoaldo Grotta-Ragazzo, Daniel Smania},pdfkeywords={  distribution, Birkhoff sums, deformation, topological classes, piecewise expanding maps}}

\begin{document}

\begin{abstract}  Often topological classes  of one-dimensional dynamical systems are finite codimension smooth manifolds.   We describe a method to prove this sort of statement  that we believe can be applied in many settings.  In this work we   will  implement it for  piecewise expanding maps. The most important step  will be  the identification of  infinitesimal deformations with primitives of Birkhoff sums (up to addition of a Lipschitz function), that allows us to use the ergodic  properties  of piecewise expanding maps to study the regularity of infinitesimal deformations. \end{abstract}

\maketitle

\setcounter{tocdepth}{1}
\tableofcontents




\section{Introduction}

One of the most interesting features of one-dimensional dynamics, either real or complex,  is that \\

\noindent {\it Often topological classes  of one-dimensional dynamical systems are smooth manifolds with finite codimension. } \\

So we can consider  topological classes as an (infinite-dimensional) Teichm\"uller space of those dynamical systems (see Gardiner and  Lakic \cite{gl} for information on classical Teichm\"uller spaces).  Given a dynamical system $f$, there are many ways to represents this topological object, with distinct geometries.  We can ask how the geometrical properties of these representations change when we deform them; that is when we consider a smooth curve inside the topological class.  We can, for instance, study the impact of deformations on periodic orbits, the Hausdorff dimension of sets, and  invariant measures (the so-called linear response problem. See, for example, Baladi \cite{lr}).

This fact is a quite useful tool to study these dynamical systems. This is  a well-developed approach in the study of {\it complex} one-dimensional dynamical systems, as   rational (and especially polynomial) maps and (quasi-)Fuchsian groups (Teichm\"uller theory).  The introduction of  holomorphic motions by Ma\~ne, Sad and Sullivan \cite{mss}  had  a significant impact  in this field. The nonexistence of rational maps with wandering domains  established by Sullivan \cite{s2} relies on  deformation methods. See also the Teichm\"uller space of rational maps by McMullen and  Sullivan \cite{msu}. In both cases, the space of deformations is finite-dimensional. 

 One of the main results on (infinite-dimensional) deformations of dynamical systems  is the study of topological classes (more precisely, hybrid classes) of quadratic-like maps by M. Lyubich \cite{lyubich} and real-analytic unimodal maps by Avila, Lyubich and de Melo \cite{alm}, which were important to the study of renormalization  and the typical behavior of such maps.  There are also recent results for topological classes of real-analytic multimodal maps by Clark and van Strien \cite{cstt}. All those results rely heavily on complex methods. 

There are also  related results for real analytic circle diffeomorphisms by E. Risler \cite{risler} (see also Goncharuk and Yampolsky \cite{nm}), generalized interval exchange transformations by Marmi, Moussa and Yoccoz \cite{mmy}, dissipative gap mappings by Clark and Gouveia \cite{cg},  piecewise expanding unimodal maps by Baladi and S. \cite{smooth},  and  piecewise Mo\"ebius circle  diffeomorphisms with a break by Khanin and Teplinsky \cite{tt}, but one can expect that topological classes are finite codimension  smooth manifolds in many  settings in one-dimensional dynamics. Most of these works are deeply connected with renormalization theory, once  in many settings  the {\t local stable manifolds} of the omega-limit set  of renomalization operator {\it are}  topological classes.

 We also refer to the  universal Teichm\"uller space (see the survey by Gardiner and Harvey \cite{gh}), the study of the manifold structure of normalized potentials by Giulietti,  Kloeckner,  Lopes,  and Marcon \cite{giu} and rigidity conjectures by Martens,  Palmisano and Winckler \cite{mp} and Winckler \cite{w}.

One may ask  if there is an {\it unified  way} to study the smoothness and finite codimension of the topological classes in one-dimensional dynamics.  Inspired by  previous works by Lyubich \cite{lyubich} and  Baladi and S.  \cite{smooth}\cite{alternative}  we will describe a method that we believe is quite general. In this work,  we   will  implement it for the class  of piecewise expanding maps.

Maps with discontinuities, non analytic critical points and/or  Lorenz-like singularities have been  quite resistant   to  complex dynamics  methods. On the other hand, the ergodic theory of those maps had a massive development in the last decades. One of the main distinctions of this work is the use of purely real methods, where {\it ergodic theory}   will play a surprisingly new and crucial role. 

\part{Heuristic of the method}

\section{Abstract nonsense results} By {\it abstract nonsense results} one must understand that the arguments here can be carried out for every  class of one-dimensional dynamical systems where assumptions A. and B. below holds. We only use soft arguments that one can adapt with minimal modifications in many settings.  Suppose we have a smooth family of (piecewise) smooth  dynamical systems $f_t\colon X \rightarrow X$, with $t$ in an open subset $O$ of some vector space, with $0\in O$. The phase space $X$ is a one-dimensional manifold (perhaps with borders). 

 Suppose that $f_t$ belongs to the topological class of $f_0$ for every $t$. We call $f_t$ a {\it smooth deformation} of $f_0$. That is, there is a family of homeomorphisms $h_t$ such that 
\begin{equation}\label{conj} h_t\circ f_0 = f_t \circ h_t\end{equation}
 Assume that this smooth deformation $f_t$ satisfies \\

\noindent {\bf Assumption A (smooth motions)}  For each  $x\in X$ we have that 
$$t\mapsto h_t(x)$$ 
is differentiable. \\

So we can derive (\ref{conj}) with respect to $t$ to obtain
\begin{equation}\partial_t h_t\circ f_0 =\partial_t f_t \circ h_t + \partial_x f_t \circ h_t \cdot \partial_t h_t. \end{equation} 
that is, applying $h_t^{-1}$ on the right 
\begin{equation}\label{tce} v_t=\alpha_t \circ f_t  -Df_t \cdot  \alpha_t(x)\end{equation}
for every t. Here $v_t=\partial_t f_t$ and $\alpha_t=  \partial_t h_t\circ h_t^{-1}.$ 

This suggest that if a function $v$ belongs to the tangent space of the topological class of a map $f$ the there is a solution $\alpha$ for the {\bf  twisted cohomological equation} 
\begin{equation}\label{tce2} v=\alpha \circ f  -Df \cdot  \alpha\end{equation}

The function $\alpha$ will be called an {\it infinitesimal deformation}  of $f$.  Note also  that $h_t$ satisfies the initial value problem 
\begin{equation}\label{ode}   \begin{cases} \partial_t h_t (x)=  \alpha_t(h_t(x)),\\
                                          h_0(x)=x.  \end{cases}\end{equation}
It is remarkable that one can often {\it solve} (\ref{tce}) without knowing $h_t$. Indeed note that the series 
$$\alpha(x)=-\sum_{i=0}^\infty \frac{v(f^i(x))}{Df^{i+1}(x)}$$
is a {\it formal} solution of (\ref{tce2}). Baladi and S. \cite{bs0}\cite{smooth} proved that for piecewise expanding unimodal maps  this series indeed converges to a H\"older solution. We are going to improve that showing  that $\alpha$ is indeed Log-Lipschitz  in the piecewise expanding setting (see Theorem \ref{infc}).  When dealing with maps with critical points, as Collet-Eckmann maps, this formal series does not converge everywhere, but nevertheless one can implement an inducing scheme to obtain  a continuous solution $\alpha$ (See Baladi and S.  \cite{bs3}).

\begin{remark} Assumption A seems to be quite strong. However, it is well-known in complex dynamics that for complex-analytic families of complex analytic maps with no bifurcations, the map $t\mapsto h_t(x)$  is {\it holomorphic } for each $x$, that is, a {\it holomorphic motion}, that has been an essential tool to study such dynamical systems since its introduction in the study of rational maps by Ma\~ne, Sad and Sullivan \cite{mss}, and in particular in the study of deformations of quadratic-like maps in Lyubich \cite{lyubich}.  The use of {\it Beltrami paths} to construction deformations of complex dynamical systems is a quite popular way to built deformations satisfying Assumption A. It is remarkable that Assumption A. also holds for smooth deformations of real maps with finite smoothness. Indeed Baladi and S.\cite{smooth} proved that Assumption A. holds for smooth families of piecewise expanding unimodal maps, and we will prove that it holds for smooth families of piecewise expanding maps (See Theorem \ref{char}).
\end{remark}

We started with a smooth family in the topological class of $f_0$ and verified that the twisted cohomological equation  (\ref{tce2}) for its tangent vector $v=\partial_t f_t|_{t=0}$ has a solution $\alpha$. We would like to do the {\it reverse argument}.  \\

Suppose $f_t$ is a smooth family,  such that (\ref{tce}) has a solution $\alpha_t$ for every $t$. Then we can consider the initial value problem (\ref{ode}). If {\it $\alpha_t$ is regular enough}, this problem is uniquely integrable and it defines a flow $h_t$. One can see that due (\ref{tce}) that   $h_t\circ f_0$ and $f_t\circ h_t$ are both solution of the initial value problem
\begin{equation*}\label{ode2}   \begin{cases} \dot{y}=  \alpha_t(y),\\
                                          y(0)=f_0(x).  \end{cases}\end{equation*}
and consequently $h_t\circ f_0=f_t\circ h_t$, so $h_t$ is a conjugacy between $f_t$ and $f_0$ and $f_t$ belongs to the topological class of $f_0$ for every $t$. 

So the reverse argument needs the following assumption on the smooth family $f_t$    \\                                    
                                          
\noindent {\bf Assumption B (unique integrability)}  If $\alpha_t$ are solutions of  (\ref{tce2}) then the  ordinary differential equation $$\dot{y}=\alpha_t(y)$$ is uniquely integrable. \\

Note that we {\it do not } assume that the smooth family $f_t$ is a deformation, but we {\it conclude} this from Assumption B. \\ 

\begin{remark} For complex analytics dynamical systems  we often have that $\alpha_t$ are {\it quasiconformal vector fields}, that implies the unique integrability (See McMullen \cite{mcmullen}). Baladi and S. \cite{smooth} showed that for families of piecewise expanding unimodal maps $f_t$ the unique integrability of this  o.d.e. holds provided there  are {\it continuous solutions}  $\alpha_t$. \end{remark}

This heuristic suggests that \\

\noindent {\it   If the topological class of map $f$ is a smooth manifold, then the vectors $v$ in the tangent space of $f$ are those that admits solutions $\alpha$ of the twisted cohomological equation (\ref{tce2})  that are regular enough to warranty that the ordinary differential  equation (\ref{ode}) is uniquely integrable. Moreover, if the topological class has finite codimension $d$ then the  subspace of vectors $v$ that admits such regular solutions has codimension $d$.}

\section{Conditions for unique integrability} \label{rem1} In the setting of this work the unique integrability of  (\ref{ode})  does not follow from the usual Picard-Lindelöf-Cauchy-Lipschitz theorem since typically  the solutions $\alpha$ on (\ref{tce2}) are not Lipchitz functions even when $f$ and $v$ are very smooth. 
For instance, if  $f$ is a $C^\infty$ (or even analytic)  expanding map of the circle and $v \in C^\infty$ (or even analytic), we have that the solution $\alpha$ is often {\it nowhere differentiable} and indeed  it is not Lipchitz on any subset of the circle with positive one-dimensional Haar measure. See de Lima and S.  \cite{central}. \\

There are two ways to obtain unique integrability. The first is to use a property similar to the sensibility of the initial conditions of $f_t$. This was done in Baladi and S. \cite{smooth} for piecewise expanding unimodal maps. The downside of  this approach is that it does not give us  information about the regularity of the conjugacies $h_t$. The second approach, that we adopted here, is to prove that the solution $\alpha$ satisfies the Osgood condition \cite{osgood}, which implies unique integrability. \\

Indeed we are going to show that in the piecewise expanding setting $\alpha$ is Log-Lipschitz, which implies Osgood condition. See Theorem \ref{infc}.  Moreover, the Log-Lipschitz continuity of $\alpha$  implies that the conjugacy $h_t$ and its inverse are  $1-O(|t|)$-H\"older continuous by a result by Chemin \cite{chemin}, who was interested in the regularity of the flow generated by vectors that are weak solutions of the $2D$ Euler equations. Indeed there are even more recent results relating the modulus of continuity of $\alpha$ and the modulus of the continuity of the flow (See Kelliher \cite{kelliher}). If $\alpha_t$ is Zygmund, then the conjugacies are $1+O(|t|)$-quasisymmetric by Riemann \cite{reimann}. See Table 1.

{\small
\begin{table*}[h]
    \centering
        \begin{tabular}{c c c c}
            \toprule
            \midrule
                 {\tiny Regularity of $\alpha_t$ } & \multirow{1}{*}{\parbox{1.8cm}{\tiny unique  \\ integrability?} } & \multirow{1}{*}{\parbox{1.8cm}{\tiny Regularty of  \\ $h_t$} }& \multirow{1}{*}{\parbox{1.8cm}{\tiny  Regularity   \\ for small $|t|$ }}  \\ \cmidrule{1-4}
                             \multicolumn{1}{l}{\tiny Continuity}  & 
                \multicolumn{1}{l}{\tiny Not Always}&   &  \\
                  \cmidrule{1-4} 
                                      \multicolumn{1}{l}{\tiny $\beta$-H\"older, $\beta\in (0,1)$}  &
                      \multicolumn{1}{l}{\tiny Not Always} & &    \\
                    \cmidrule{1-4}
                                        \multicolumn{1}{l}{\tiny Log-Lipschitz}    &
                \multicolumn{1}{l}{\tiny Yes}& {\tiny H\"older}  & {\tiny $\frac{|h_t(x+\delta)-h_t(x)|}{|\delta|^{1-O(|t|)}}=1+O(|t|)$}   \\
                  \cmidrule{1-4}
                                        \multicolumn{1}{l}{\tiny Zygmund}    &
                \multicolumn{1}{l}{\tiny Yes} & {\tiny Quasisymmetric}  & {\tiny $\frac{|h_t(x+\delta)-h_t(x)|}{|h_t(x-\delta)-h_t(x)|}=1+O(|t|)$ } \\
                                 \cmidrule{1-4}
                                           \multicolumn{1}{l}{\tiny Lipschitz}    &
                \multicolumn{1}{l}{\tiny Yes}&  {\tiny Lipschitz}  & {\tiny $\frac{|h_t(x+\delta)-h_t(x)|}{|\delta|}=1+O(|t|)$}  \\
            \midrule
            \bottomrule
        \end{tabular}
        \centering
        \vspace{3mm}
        \caption{Regularity of $\alpha_t$ versus regularity of conjugacies. }
        \label{tab:sam_count}
\end{table*}
}

\begin{remark} For piecewise expanding maps,  the conjugacies $h_t$ are typically  not very regular. Indeed, Shub and Sullivan \cite{ss} show that if  the conjugacy between two expanding maps on the circle is absolutely  continuous, then it is indeed smooth. In particular, the conjugacy must preserve the multipliers of the periodic points. Since it is easy to see that there are deformations $f_t$  that do not preserve multipliers, $h_t$ is rarely absolutely continuous.  A large class of unimodal maps have similar properties (See Martens and de Melo \cite{mm}).  This, in particular, suggests that Lipschitz regularity of  $\alpha_t$ is quite rare for one-dimensional maps  with many periodic points. \end{remark} 

\section{The heart of the paper: Derivatives of infinitesimal conjugacies} 

So the  crucial step in the above  method demands the  study of  the existence and regularity of the solution $\alpha$ for (\ref{tce2}).  Remark \ref{rem1} tells us we can not expect $\alpha$ to be very regular. However  we can {\it formally} derive (\ref{tce2}) to obtain
\begin{equation*} Dv= D\alpha\circ f  \cdot Df -  D^2f\cdot \alpha - Df\cdot D\alpha \end{equation*} 
so $D\alpha$ satisfies the Livsic cohomological equation
\begin{equation}\label{livsic}  \frac{Dv+ D^2f\cdot \alpha}{Df}  = D\alpha\circ f - D\alpha. \end{equation} 
If we denote 
\begin{equation}\label{phi}  \phi =  \frac{Dv+  D^2f\cdot \alpha}{Df} \end{equation} 
then we can {\it formally} solve (\ref{livsic}) taking 
\begin{equation} \label{solliv} D\alpha =- \sum_{i=0}^\infty \phi\circ f^i.\end{equation} 
We are going to see that one can make this argument rigorous in the one-dimensional piecewise expanding setting. 
We will prove that the derivative of $\alpha$ is indeed a Birkhoff sum {\it in the sense of distributions} (up to the addition of a bounded function) and this will allows us to study the regularity of $\alpha$. See Theorem \ref{tt}, it is heart of this work. 

\section{Perturbation  of lyapunov exponents} \label{lya}  The role of Birkhoff  sums of  the observable $\phi$ in (\ref{phi}) is clarified when we study the perturbation of the lyapunov exponent along orbits in a deformation $f_t$ of $f_0$. If $h_t\circ f_0 = f_t\circ h_t$ then for a given $x$ we have that $h_t(x)$ is a "smooth" continuation of $x$ and we can see that
\begin{align*} \partial_t \ln |Df^k_t(h_t(x))|\Big|_{t=0}&= \sum_{j< k} \partial_t \ln |Df_t(h_t(f_0^j(x)))|\Big|_{t=0}\\
&= \sum_{j< k}  \phi(f_0^j(x)).\end{align*} 

\part{Piecewise expanding maps} 

\section{The class of piecewise expanding maps}

Let $I=[a,b]$ and  $C=\{  c_0, c_1, \dots, c_n\} \subset [a,b]$ be such that $a=c_0< c_{i} <c_{i+1}< c_n=b$ for every $i< n-1$. Denote by $\mathcal{B}^k(C)$, with $k\in \mathbb{R}$, $k\geq 0$, the  space of all functions 
$$w\colon \bigcup_{i=0}^{n-1}(c_i,c_{i+1})\rightarrow \mathbb{R}$$
such that for  each $i< n-1$ we have that $w\colon (c_i,c_{i+1})\rightarrow\mathbb{R}$ extends to a $C^m$ function  in $[c_i,c_{i+1}]$, where $k=m+\beta$, with $n\in \mathbb{N}$ and $\beta\in [0,1)$ and, in the case $\beta\neq 0$, a  $\beta$-H\"older $n$th derivative.   We can endowed $\mathcal{B}^k(C)$  with the norm
$$|w|_k= \sup_{i< n}  \sum_{j\geq m}  |D^j w|_{L^\infty[c_i,c_{i+1}]}+ \sup_{\substack{x,y\in [c_i,c_{i+1}]\\ x\neq y}} \frac{|D^mf(x)-D^mf(y)|}{|x-y|^\beta}.$$
We have that $\mathcal{B}^k(C)$ is a Banach space. Denote by  $\mathcal{B}^k_{exp}(C)$, with $k\geq 1$,  the set of all $f\in  \mathcal{B}^k(C)$ such that the range of $f$ is contained in $[a,b]$ and $\inf |Df|> 1$.  This is a convex subset of  $\mathcal{B}^k(C)$. Note that if $f \in \mathcal{B}^k_{exp}(C)$ then $f^i \in \mathcal{B}^k_{exp}(C_i)$ for some finite set $C_i$ that depends on $f$.

Let 
$$\hat{I}= (a,b) \times \{+,-\} \ \cup \{ (c_i,+),(c_i,-)\colon 1\leq i\leq n-1   \}.$$
To simply the notation we will use $x^+$ instead of $(x,+)$ and $x^-$ instead of $(x,-)$.

Suppose that $f\in \mathcal{B}^k(C)$ is  piecewise monotone function on each interval $(c_i,c_{i+1})$. This is the case for $f\in \mathcal{B}^k_{exp}(C)$. Then we can extend it  to a function 
$$f\colon \hat{I}\rightarrow \hat{I}$$
using the lateral limits of $f$. For instance we define $f(a^+)=y^-$ when $$\lim_{x\rightarrow a^+}f(x)=y^-.$$
A function $v\in \mathcal{B}^k(C)$ can be extended to a function
$$v\colon \hat{I}\rightarrow \mathbb{R}$$
as $v(a^+)=\lim_{x\rightarrow a^+} v(x)$ and $v(a^-)=\lim_{x\rightarrow a^-} v(x)$. Note that if $f\in \mathcal{B}^k_{exp}(C)$ then $D^if\in \mathcal{B}^{k-i}(C)$ for every $i\leq k$.

We say that $f,g \in \mathcal{B}^k_{exp}(C)$ are topologically conjugated by  a homeomorphism $h\colon I\rightarrow I$ if $h \circ  f = g\circ h$ on  $\hat{I}$. Note that the conjugacy $h$ is always H\"older continuous (Buzzi \cite{buzzi} has an elegant proof of this for an even more general setting).

\subsection{Transfer operator and Lasota-Yorke inequality}\label{layo}  Let $f\in \mathcal{B}^2_{exp}(C)$ and $L$ be the transfer operator of $f$ with respect to the Lebesgue measure $m$ on $I$. A well-known result by Lasota and Yorke \cite{ly} tell us that 

\begin{itemize}
\item{\bf (Lasota-Yorke inequality in BV)}  There exist  $\Cll{b}$ and $\Cll[c]{c}$ such that  
\begin{equation}\label{ly} |L \gamma|_{BV}\leq    \Crr{c} |\gamma|_{BV}+ \Crr{b}|\gamma|_{L^1},\end{equation}
\end{itemize} 
So 
\begin{equation}\label{lyy} |L^i \gamma|_{BV}\leq  \Cll{ff} \Crr{c}^i |\gamma|_{BV}+ \Cll{0}|\gamma|_{L^1}. \end{equation}
for every $i$. Denote 
$$\Lambda_f = \{ \lambda\in \mathbb{S}^1\colon \ \lambda \in \sigma(L)  \},$$
where $ \sigma(L)$ denotes  the spectrum of $L$ acting on  $BV$. Then  $\Lambda_1$ if finite,$1\in \Lambda_1$  and 
\begin{equation}\label{ii}  L = \sum_{\lambda\in \Lambda_1}   \lambda \Phi_\lambda   +    K.\end{equation}
Here $\Phi_\lambda^2=\Phi_\lambda$, $\Phi_\lambda \Phi_{\lambda'}=0$ if $j\neq j'$ and $K\Phi_\lambda=\Phi_{\lambda}K=0$. Furthermore 
\begin{itemize}
\item[i.]  $\Phi_\lambda$ is  finite rank projection  and it  has a extension as a  bounded linear transformation $\Phi_\lambda\colon L^1(m)\rightarrow BV$, 
\item[ii.] $K$ is a bounded linear operator in $BV$ whose  spectral radius  is smaller than one, that is, there is $\Cll[c]{cpf}  \in (0,1)$ and $\Cll{cd}$ satisfying 
$$|K^n(\phi)|_{BV}\leq \Crr{cd}\Crr{cpf}^n |\phi|_{BV}.$$
\item[iii.]  there exits  $p=p(f)\in \mathbb{N}^\star$ such that for all  $\lambda\in \Lambda_1$ we have $\lambda^p=1$. \\
\end{itemize} 

Let $\Cll[c]{cmin}= sup_{x\in I} |Df(x)|^{-1}.$ We will also need
\begin{lemma} \label{bvpartition} There exist $\Cll{dist} >0$ and $\Cll{distp} > 0$  such that  for all $n\in \mathbb{N}$ and $J\in \mathcal{P}^n$  
\begin{equation}\label{ddis}   \frac{1}{ \Crr{dist}} \leq  \frac{ Df^n(x)}{Df^n(y)}\leq \Crr{dist}\end{equation} 

\begin{equation}\label{ddis2}  |\ln  |Df^n(x)|  - \ln |Df^n(y)| |\leq \Crr{distp}|f^n(x)-f^n(y)|.\end{equation} 
for every  $x,y\in J$. Furthermore  if $\gamma$ is a function with  bounded variation and its support is contained in $J$ we have 
\begin{equation}\label{bvest}  v(L_F^n(\gamma))\leq \frac{\Crr{dist} }{|DF^n(x)|} \Big( v(\gamma)+  \Crr{distp} |I|   |\gamma|_{L^\infty(m)}\Big),\end{equation} 
that for every $x\in J$ , where $v(g)$ is  the variation of $g$. 
\end{lemma}

\section{ Regularity of infinitesimal  conjugacies} 

\begin{theorem}\label{tt} Let $f\in \mathcal{B}_{exp}^{2+\beta}(C)$, with $\beta\in[0,1)$,  and  $p=p(f)$. Let $\alpha\colon I \rightarrow \mathbb{C}$ be a continuous function. The following statements are equivalent 
\begin{itemize} 
\item[A.] There is a   function  $v\in \mathcal{B}^{1+\beta}(C)$ such that 
$$v(x)= \alpha(f(x))-Df(x)\alpha(x)$$
for every $x\not\in Crit(f)$.
\item[B.]  We have that 
$$\alpha(x)= H(x)+ G(x)+ \int1_{[a,x]}  \Big( \sum_{n=0}^{\infty}  \sum_{i=0}^{p-1}  \phi\circ f^{np+i} \Big) dm,$$
where 
\begin{itemize} 
\item $\phi\in\mathcal{B}^\beta(C)$  satisfies
\begin{equation}\label{kl} \int \phi\Phi_1(\psi) \ dm=0\end{equation} 
for every $\psi\in BV$,  
\item $H$ is a Lipschitz  function such that $DH\circ f - DH$ belongs to  $\mathcal{B}^\beta(C)$ and 
\item $G$ is given by 
$$G(x)= \int \phi \Big( \sum_{\lambda\in \Lambda_1\setminus\{1\}}  \frac{1}{1-\lambda}  \Phi_\lambda(1_{[a,x]})    \Big) \ dm$$
 if $\Lambda_f\setminus\{1\} \neq \emptyset$, or zero otherwise. Here $\Phi_\lambda$ are the projections defined in Section \ref{layo}. $G$  is also a Lipschitz function.
\end{itemize}
\end{itemize} 
\vspace{5mm}
Indeed we can choose
 $$\phi = \frac{Dv + D^2f \alpha }{Df }.$$
\end{theorem}

%


\begin{proof}[Proof of Theorem \ref{tt} ($B\implies A$)]  Define
$$\tilde{\alpha}_n(x)=\int   1_{[a,x]}\Big( \sum_{k=0}^{n}\phi\circ f^{k}\Big)\  dm$$
and
$$\hat{\alpha}_u(x)= \frac{1}{u} \sum_{n=0}^{u-1} \tilde{\alpha}_n(x).$$
By G.R. and S.  \cite[Theorem \ref{sum-conv2}]{um}  the following (uniform) limit
$$
\beta(x)=\lim_u \hat{\alpha}_u(x)=  \int 1_{[a,x]} \sum_{n=0}^\infty \sum_{j=0}^{p-1} \phi\circ f^{pn+j} \ dm + G(x),
$$
exists for every $x\in I$ and $\beta$ is a Log-Lipschitz function, where
$$G(x)= \int \phi \Big( \sum_{\lambda\in \Lambda_1\setminus\{1\}}  \frac{1}{1-\lambda}  \Phi_\lambda(1_{[a,x]})    \Big) $$
is a Lipchitz function.    Define
$$w_n(x)= \hat{\alpha}_n(f(x))-Df(x)\hat{\alpha}_n(x)$$
and
$$w(x)=\beta(f(x))-Df(x)\beta(x).$$
Since $\tilde{\alpha}_n$ is continuous and piecewise smooth,   for every right continuous function $\psi\in BV$  such that 
$$supp \ \psi \subset [c_i,c_{i+1}]$$
for some $i < n$, we have 
\begin{align} & -\int w_n \ d\psi = -\int \big(\tilde{\alpha}_n\circ f -Df \tilde{\alpha}_n \big)\ d\psi \nonumber \\
&= w_n(a^+)\psi(a)+ \int D\tilde{\alpha}_n \circ f Df  \psi   - \Big( D^2f  \tilde{\alpha}_n + Df D\tilde{\alpha}_n  \Big)\psi \ dm\nonumber \\ 
&= w_n(a^+)\psi(a) +  \int   Df  \ \psi \sum_{k=0}^n \phi\circ f^{k+1} - \Big( D^2f \tilde{\alpha}_n + Df \sum_{k=0}^n \phi\circ f^{k}   \Big)\psi \ dm\nonumber \\ 
&=  w_n(a^+)\psi(a) + \int   \Big(- D^2f \tilde{\alpha}_n - Df \phi  +   Df \phi\circ f^{n+1}   \Big)\psi \ dm,\nonumber 
\end{align} 
so
\begin{align} & \ -\int \big(\hat{\alpha}_u\circ f-Df \hat{\alpha}_u\big)\ d\psi  \\
= \psi(a)  \frac{1}{u} \sum_{n=0}^{u-1} w_n(a^+)  &+\int   \Big(- D^2f \hat{\alpha}_u - Df \phi \Big)\psi  \ dm + \frac{1}{u}  \int  \psi Df \sum_{n=0}^{u-1} \phi\circ f^{n+1} \ dm\nonumber 
\end{align} 
Due (\ref{kl}) we have 
$$\lim_u  \int \frac{1}{u}  \psi Df \sum_{n=0}^{u-1} \phi\circ f^{n+1} \ dm =0,$$
so we conclude that 
\begin{align} & \ -\int w \ d\psi \nonumber =w(a^+)\psi(a)+ \int   \Big(- D^2f  \beta - Df \phi  \Big)\psi \ dm.\nonumber 
\end{align} 
Taking $\psi=1_{[c_i,x]}$ we obtain 
$$w(x)= w(c_i^+) + \int_{c_i}^x  - D^2f  \beta  - Df \phi \ dm.$$
for every $x\in [c_i,c_{i+1}]$, so  $w$ is a piecewise $C^{1+\beta}$ function.  Define
$$s= H\circ f -Df \cdot H.$$
Then
$$Ds =Df (H'\circ f-H')- D^2f  H,$$
 so $s$ is  a piecewise $C^{1+\beta}$ function due the assumptions on $H$. So $\alpha=H + \beta$ satisfies
 $$w+s = \alpha\circ f -Df \alpha,$$
 where $v=w+s$ is a piecewise $C^{1+\beta}$ function. 
\end{proof}

 \begin{lemma}\label{aad}  Let $v \in \mathcal{B}^{1+\beta}(C) $, with $\beta\in [0,1)$,  such that there is a continuous function  $\alpha:[a,b]\rightarrow \mathbb{R}$ that satisfies the equation
\begin{equation}\label{pp}  v(x)= \alpha(f(x))-Df(x)\alpha(x).\end{equation} 
for every $x\not\in C$. 
Then there exists $\Cll{neww}$ such that for every $j$ there is a $C^\infty$ function  $\theta_j\colon I \rightarrow \mathbb{C}$ satisfying 
\begin{enumerate} 
\item[A.] The function $$\alpha_j(x)= \sum_{k=0}^{j-1}\frac{v(f^k(x))}{Df^{k+1}(x)}+\frac{\theta_j(f^{j}(x))}{Df^{j}(x)}.$$
has a continuous extension to $I$, 
\item[B.] $|D\theta_j|_{L^1(m)}\leq \Crr{neww}  j,$
\item[C.]  $|\theta_j|_{L^\infty(m)}\leq  \Crr{neww} ,$
\item[D.]  We have that
$$\sup_{\lambda\in \Lambda_1} \sup_j \sup_{\psi\in BV, \psi \neq 0} \frac{|\int D\theta_j  \Phi_\lambda(\psi) \ dm |}{|\psi|_{L^1(m)}} \leq  \Crr{neww} .$$
\item[E.] $\alpha_j(c)=\alpha(c)$ for every $c\in C$. 
\end{enumerate} 

\end{lemma}
\begin{proof} Note  that the continuity of $\alpha$ and (\ref{pp}) implies that 
\begin{equation}\label{ppp}  \alpha(c)= - \frac{v(c^\pm)-\alpha(f(c^\pm))}{Df(c^\pm)}.  \end{equation} 
for every $c\in C_f$. 
 Let $h\colon \mathbb{R}\rightarrow \mathbb{R}$ be a $C^\infty$ function such that $h(0)=1$ and $h(x)=0$ for $x\not \in (-1,1)$. Let $j\in \mathbb{N}$.  For $c\in \hat{C}$ denote
 $$\mathcal{O}^+(c,j)=\{ x\in I\colon x=f^i(c), \ i\leq j\}$$
 and
 $$\mathcal{O}^+_f(j)=\cup_{c\in \hat{C}} \mathcal{O}^+(c,j).$$
Let
$$\theta_j(x)=\sum_{y \in \mathcal{O}^+_f(j)}   \alpha(y)h(M(x-y)).$$
We will choose $M$ later. Since $\alpha$ is bounded we have that  $|D\theta_j|_{L^1}\leq C j$. Moreover if $M$ is large enough we have  $\theta_j(x)=\alpha(x)$ for $x \in \mathcal{O}^+_f(j)$, and   $|\theta_j|_{L^\infty}\leq C$.  To prove $D.$, consider $\lambda\in \Lambda_1$. Then the image $E_\lambda$ of $\Phi_\lambda$ has finite dimension.  Let $w_1, \dots, w_d \in BV$ be a basis  of $E_\lambda$. Then
\begin{equation} \label{k0} \Phi_\lambda(\psi)= \sum_{i=1}^d   r_i(\psi)w_i,\end{equation}
where $r_i\in L^1(I)^\star$. Since  $w_i$ has bounded variation we have 
$$\sum_{x\in \mathcal{O}^+_f(j)} |w_i(x^+)-w_i(x^-)|\leq |w_i|_{BV}$$
for every $j$.  For every $j$   we can choose $M$ large enough such that
$$\sum_{y\in \mathcal{O}^+_f(j)} |w_i(z^+_y)-w_i(y^+)|+ |w_i(z^-_y)-w_i(y^-)|\leq 1,$$
for all $z^-_y \in [y-1/M,y]$ and $z^+_y\in [y,y+1/M]$, and moreover
$$[y-1/M,y+1/M]\cap   [z-1/M,z+1/M] =\emptyset $$
for $y\neq z$ with  $y,z\in  \mathcal{O}^+_f(j)$. Since 
$$supp \ \theta_j \subset \cup_{y \in \mathcal{O}^+_f(j)} [y-1/M,y+1/M],$$
and furthermore 
$$ \int_{y-1/M}^y  M Dh(M(x-y)) \ dm(x)=1,  \  \int_y^{y+1/M}   M Dh(M(x-y)) \ dm(x)=-1$$
and
$$ \int_{y-1/M}^{y+1/M}  |M Dh(M(x-y))| \ dm(x) =  \int_{-1}^{1}  |Dh(x)| \ dm(x)$$
we obtain
\begin{align*}& \int D\theta_j  w_i \ dm =   \sum_{y \in \mathcal{O}^+_f(j)}  \int_{[y-1/M,y+1/M]} D\theta_j  w_i \ dm  \ dm  \\   &=\sum_{y \in \mathcal{O}^+_f(j)}   \alpha(y)  \int_{y-1/M}^y  M Dh(M(x-y))  w_i(x) \ dm(x)\\
&+\sum_{y \in \mathcal{O}^+_f(j)}  \alpha(y) \int_y^{y+1/M}   M Dh(M(x-y))  w_i(x) \ dm(x) \\
&=\sum_{y \in \mathcal{O}^+_f(j)}   \alpha(y)  \int_{y-1/M}^y  M Dh(M(x-y)) [ w_i(y^-)  +  (w_i(x)-w_i(y^-))]   \ dm(x)  \\
&+\sum_{y \in \mathcal{O}^+_f(j)}  \alpha(y) \int_y^{y+1/M}   M Dh(M(x-y)) [ w_i(y^+)  +  (w_i(x)-w_i(y^+))]  \ dm(x) \\
&=\sum_{y \in \mathcal{O}^+_f(j)}   \alpha(y)w_i(y^-)   +  \int_{y-1/M}^y  M Dh(M(x-y)) (w_i(x)-w_i(y^-)) \ dm(x)  \\
&+\sum_{y \in \mathcal{O}^+_f(j)}  -\alpha(y)w_i(y^+) +  \int_y^{y+1/M}  M Dh(M(x-y)) (w_i(x)-w_i(y^+))  \ dm(x) \\
&=\sum_{y \in \mathcal{O}^+_f(j)}   \alpha(y)(w_i(y^-) -w_i(y^+)) + Q(j,i,M),
\end{align*} 
with 
\begin{align*} &|Q(j,i,M)|\\
&\leq |Dh|_{L^1(m)}\Big(\sum_{y \in \mathcal{O}^+_f(j)} \sup_{x\in [y,y+1/M]} |w_i(x)-w_i(y^+)| + \sup_{x\in [y-1/M,y]}  |w_i(x)-w_i(y^-)|  \Big) \\
&\leq  |Dh|_{L^1(m)}.
\end{align*}
Moreover
$$|\sum_{y \in \mathcal{O}^+_f(j)}   \alpha(y)(w_i(y^-) -w_i(y^+))|\leq \sup_{x\in I} |\alpha(x)| |w_i|_{BV}.$$
so we conclude that
\begin{equation}\label{k3} \Big| \int D\theta_j  w_i \ dm \Big|\leq   \sup_{x\in I} |\alpha(x)| |w_i|_{BV} +  |Dh|_{L^1(m)}.\end{equation} 
By (\ref{k0}) and (\ref{k3}) we obtain $D.$

 Define recursively the functions $\alpha_{j,k}$, $k\leq j$,  as $\alpha_{j,0}=\theta_j$ and
$$\alpha_{j,k+1}(x)=-\frac{v(x)}{Df(x)} +\frac{\alpha_{j,k}(f(x))}{Df(x)}$$
for $x\not\in C_f$.  One can easily see that due (\ref{ppp}) the function  $\alpha_{j,k}$ has a continuous extension to $I$ and
$\alpha_{j,k}(x)=\alpha(x)$ for $x\in  \mathcal{O}^+_f(j-k)$.   It follows that 
$$\alpha_{j,\ell}(x)=- \sum_{k=0}^{\ell-1} \frac{v(f^k(x))}{Df^{k+1}(x)}+\frac{\theta_j(f^{\ell}(x))}{Df^{\ell}(x)}.$$
Take $\alpha_j=\alpha_{j,j}$.
\end{proof}

Let 
$$\hat{\alpha}_j(x)=-\sum_{k=0}^{j-1} \frac{v(f^k(x))}{Df^{k+1}(x)},$$
and
$$\alpha_j(x)=\hat{\alpha}_j(x)+ \frac{\theta_j(f^{j}(x)}{Df^{j}(x)},$$
where $\theta_j$ is given by Lemma \ref{aad}. Note that for every $x \not\in \cup_{i=0}^\infty f^{-i}C_f$ we have 
\begin{equation}\label{est3} \alpha(x)-\alpha_j(x) = O(\Crr{cmin}^j),\end{equation} 
and
\begin{equation}\label{est4} \alpha(x)-\hat{\alpha}_j(x)=O(\Crr{cmin}^j),\end{equation} 
Denote 
\begin{align*} R_j(x)= D \Big( \frac{\theta_j(f^{j}(x)}{Df^{j}(x)}  \Big)=D\theta_j(f^{j}(x))- \frac{\theta_j(f^{j}(x))}{Df^{j}(x)}\sum_{n=0}^{j-1} \frac{D^2 f(f^n(x))}{Df(f^n(x))}.\end{align*}

In this work all Riemann-Stieltjes integrals of the form
$$\int \eta \ d\gamma,$$ where 
$\eta$ and $\gamma$ are right continuous $BV$ functions on $I=[a,b]$, must be considered as an Riemann-Stieltjes integral on  $(a,b]$, that is 
$$ \int \eta \ d\gamma = \int_{(a,b]} \eta \ d\nu,$$
where $\nu$ is the signed measure  on $I$ given by $\nu([a,x])=\gamma(x).$  See  for instance Revuz and Yor \cite{parts} for information on   integration by parts for Riemann-Stieltjes integrals of BV functions.

\begin{lemma} \label{limi}  Let $v$ and $\alpha$ be as in Lemma \ref{aad}. Then for  every  right continuous $\psi \in BV$ such that $supp \  \psi \subset I$  we have 
\begin{align*} 
& \int \alpha \ d\psi =  - \alpha(a)\psi(a) \\ &+ \lim_j -\int R_j \psi \ dm + \sum_{n=0}^{j-1}   \int \Big( \frac{Dv+D^2f \ \hat{\alpha}_{j-n}}{Df} \Big) L^n \psi \ dm. 
 \end{align*} 
\end{lemma}

\begin{proof}
 For every $x \not\in \cup_{i=0}^j f^{-i}C_f$ 
\begin{align*} &D\alpha_j(x)\\
&= R_j(x)- \sum_{k=0}^{j-1} \Big( \frac{Dv(f^k(x))}{Df(f^k(x))}-  \frac{v(f^k(x))D^2f^{k+1}(x)}{(Df^{k+1}(x))^2}\Big)\\
&=   R_j(x)- \sum_{k=0}^{j-1} \Big(   \frac{Dv(f^k(x))}{Df(f^k(x))} - \sum_{n=0}^k  \frac{v(f^k(x))D^2f(f^n(x))}{Df^{k-n+1}(f^{n}(x))Df(f^n(x))}\Big) \\
&=  R_j(x)-   \sum_{n=0}^{j-1}  \frac{Dv(f^n(x))}{Df(f^n(x))} -   \sum_{k=0}^{j-1}  \sum_{n=0}^k  \frac{v(f^k(x))D^2f(f^n(x))}{Df^{k-n+1}(f^{n}(x))Df(f^n(x))} \\
&=  R_j(x)-  \sum_{n=0}^{j-1} \frac{Dv(f^n(x))}{Df(f^n(x))} -  \sum_{n=0}^{j-1}  \sum_{k=n}^{j-1}  \frac{v(f^k(x))D^2f(f^n(x))}{Df^{k-n+1}(f^{n}(x))Df(f^n(x))} \\
&=   R_j(x)- \sum_{n=0}^{j-1}   \frac{Dv(f^n(x))}{Df(f^n(x))}-  \sum_{n=0}^{j-1}  \frac{D^2f(f^n(x))}{Df(f^n(x))}  \sum_{k=n}^{j-1}  \frac{v(f^{k-n}(f^n(x)))}{Df^{k-n+1}(f^{n}(x))} \\
&=  R_j(x)- \sum_{n=0}^{j-1}  \frac{Dv(f^n(x))+D^2f(f^n(x))\hat{\alpha}_{j-n}(f^n(x))}{Df(f^n(x))}.\end{align*} 
Since  $\alpha_j$ is continuous  and piecewise smooth,   and $\psi$ has  bounded variation we have
\begin{align*} 
&- \int \alpha \ d\psi =-\lim_j  \int  \alpha_j  d\psi= \lim_j  \alpha_j(a)\psi(a) + \int  D \alpha_j  \psi \ dm  \\
 &=  \alpha(a)\psi(a) + 
 \lim_j \int R_j \psi \ dm - \sum_{n=0}^j   \int \Big( \frac{Dv\circ f^n+D^2f\circ f^n \ \hat{\alpha}_{j-n}\circ f^n}{Df\circ f^n} \Big) \psi \ dm \\
 &=  \alpha(a)\psi(a)  + \lim_j \int R_j \psi \ dm - \sum_{n=0}^j   \int \Big( \frac{Dv+D^2f \ \hat{\alpha}_{j-n}}{Df} \Big) L^n \psi \ dm. 
 \end{align*} 
 \end{proof} 
 
 \begin{lemma}\label{uy} Let $v$ and $\alpha$ be as in Lemma \ref{aad}. For every  $\psi \in BV$ we have
  \begin{align} \label{est2222}
 &\sum_{n=0}^{j-1}  \int  \Big[ \Big( \frac{Dv+D^2f\hat{\alpha}_{j-n}}{Df} \Big) -    \Big( \frac{Dv+D^2f\alpha}{Df} \Big)\Big] L^n \psi \ dm \\
  = \sum_{n=0}^{j-1}  & \int  \Big[ \Big( \frac{Dv+D^2f\hat{\alpha}_{j-n}}{Df} \Big) -    \Big( \frac{Dv+D^2f\alpha}{Df} \Big)\Big]  \Big(\sum_{\lambda\in \Lambda_1}  \lambda^n \Phi_\lambda(\psi)   \Big)   \ dm+ O(j\Crr{cpf}^j). \nonumber
   \end{align} 
\begin{equation}\label{est22}  \int (R_j-D\theta_j\circ f^{j} ) \psi \ dm =O(j\Crr{cmin}^j),\end{equation} 
\begin{equation} \label{est222}
  \int D\theta_j\circ f^{j} \psi  \ dm=\int D\theta_j  \Big(\sum_{\lambda\in \Lambda_1}  \lambda^{j} \Phi_\lambda(\psi)   \Big)\ dm   + O(j\Crr{cpf}^j).
  \end{equation} 
 \end{lemma} 
 \begin{proof}  By (\ref{est3}) we have
  \begin{align*} 
 &\sum_{n=0}^{j-1}   \int  \Big[ \Big( \frac{Dv+D^2f\hat{\alpha}_{j-n}}{Df} \Big) -    \Big( \frac{Dv+D^2f\alpha}{Df} \Big)\Big] L^n \psi \ dm \\
 &=\sum_{n=0}^{j-1}   \int  \Big[ \Big( \frac{Dv+D^2f\hat{\alpha}_{j-n}}{Df} \Big) -    \Big( \frac{Dv+D^2f\alpha}{Df} \Big)\Big]  \Big(\sum_{\lambda\in \Lambda_1}  \lambda^n \Phi_\lambda(\psi)   \Big)   \ dm\\
 &+O(\sum_{n=0}^{j-1} \Crr{cmin}^{j-n}\Crr{cpf}^n) \\
 &=\sum_{n=0}^{j-1}  \int  \Big[ \Big( \frac{Dv+D^2f\hat{\alpha}_{j-n}}{Df} \Big) -    \Big( \frac{Dv+D^2f\alpha}{Df} \Big)\Big]  \Big(\sum_{\lambda\in \Lambda_1}  \lambda^n \Phi_\lambda(\psi)   \Big)   \ dm\\
 &+ O(j \max\{ \Crr{cpf}, \Crr{cmin}\}^j).
   \end{align*} 
It is easy to see that (\ref{est22}) holds, since 
$$D\theta_j(f^{j}(x))-R_j(x)=\frac{\theta_j(f^{j}(x))}{Df^{j}(x)}\sum_{n=0}^{j-1} \frac{D^2 f(f^n(x))}{Df(f^n(x))},$$
$\theta_j$ and $D^2f/Df$ are uniformly  bounded and $f$ is expanding. 
Finally
  \begin{align*}
  &\int D\theta_j\circ f^{j} \ \psi  \ dm=  \int D\theta_j \ L^{j} \psi  \ dm\\
  &=\int D\theta_j  \Big(\sum_{\lambda\in \Lambda_1}  \lambda^{j} \Phi_\lambda(\psi)   \Big)\ dm  + |D\theta_j|_{L^1(m)}O(\Crr{cpf}^j)\\
  &=\int D\theta_j  \Big(\sum_{\lambda\in \Lambda_1}  \lambda^{j} \Phi_\lambda(\psi)   \Big)\ dm   + O(j\Crr{cpf}^j).
  \end{align*} 
   \end{proof}
   
   \begin{proof}[Proof of Theorem \ref{tt}($A\implies B$)] By Lemma \ref{uy} we have
   \begin{align}\label{ioioe}
   &-\int R_j \psi \ dm + \sum_{n=0}^{j-1}   \int \Big( \frac{Dv+D^2f\hat{\alpha}_{j-n}}{Df} \Big) L^n \psi \ dm \\
   =&-\int D\theta_j  \Big(\sum_{\lambda\in \Lambda_1}  \lambda^{j} \Phi_\lambda(\psi)   \Big)\ dm  \nonumber \\
   +& \sum_{n=0}^{j-1}   \int  \Big[ \Big( \frac{Dv+D^2f\hat{\alpha}_{j-n}}{Df} \Big) -    \Big( \frac{Dv+D^2f\alpha}{Df} \Big)\Big]  \Big(\sum_{\lambda\in \Lambda_1}  \lambda^n \Phi_\lambda(\psi)   \Big) \ dm    \nonumber \\
   +& \int  \psi  \sum_{n=0}^{j-1}   \frac{Dv\circ f^n+D^2f \circ  f^n \ \alpha\circ f^n}{Df\circ f^n}\ dm + O(j\max\{\Crr{cmin},\Crr{cpf} \}^j). \nonumber
   \end{align} 
   Define the linear functional $T_j\colon L^1(m)\rightarrow \mathbb{C}$  as
   \begin{align*}
  &T_j(\psi)=-\int D\theta_j  \Big(\sum_{\lambda\in \Lambda_1}  \lambda^{j} \Phi_\lambda(\psi)   \Big)\ dm \\
  &+\sum_{n=0}^{j-1}   \int  \Big[ \Big( \frac{Dv+D^2f \ \hat{\alpha}_{j-n}}{Df} \Big) -    \Big( \frac{Dv+D^2f\alpha}{Df} \Big)\Big]  \Big(\sum_{\lambda\in \Lambda_1}  \lambda^n \Phi_\lambda(\psi)   \Big) \ dm.
   \end{align*} 
   Due  Lemma \ref{aad}.D and (\ref{est4}) we have $\sup_j |T_j|_{(L^1(m))^\star}< \infty$.  In particular 
  
   \begin{equation}\label{iii} \sup_j |T_j(\psi)|< \infty\end{equation} 
     By Theorem \ref{limi} we have that 
   \begin{align} \label{uuu}
 \int \alpha \ d\psi 
 &= -\alpha(a)\psi(a)+ \lim_j - \int R_j \psi \ dm + \sum_{n=0}^j   \int \frac{Dv+D^2f\hat{\alpha}_{j-n}}{Df} L^n \psi \ dm. 
 \end{align} 
Let
$$\phi =-\frac{Dv+D^2f \alpha}{Df}.$$
 Note that (\ref{ioioe}), (\ref{iii}) and (\ref{uuu}) imply that 
 $$\{ \int  \psi  \sum_{n=0}^{j -1} \phi\circ f^n \ dm\colon j\in \mathbb{N}\}$$
 is a bounded set for each $\psi\in BV$. Since
\begin{align*} \int  \psi  \sum_{n=0}^{j -1} \phi\circ f^n \ dm&= \int \phi \Big[\sum_{n=0}^{j-1} \Big( \sum_{\lambda\in \Lambda_1} \lambda^n \Phi_\lambda(\psi) + K^n(\psi) \Big)  \Big] \ dm \\
&= \int \phi \Big[ j \Phi_1(\psi)+ \sum_{\lambda\in \Lambda_1\setminus\{1\}}\frac{1- \lambda^{j}}{1- \lambda}   \Phi_\lambda(\psi)   \Big] \ dm +O(1)\\
&= j  \int \phi \Phi_1(\psi) \ dm + O(1),
 \end{align*} 
 this only occurs if 
$$\int \phi \Phi_1(\psi) \ dm =0$$
for every $\psi\in BV$. So G.R. and S. \cite[Lemma \ref{sum-vvv}]{um}  implies that  the limit
 $$\lim_j \int \psi \cdot  \Big( \sum_{n=0}^{pj}  \phi\circ f^n \Big) dm$$
exists. So the limit $T(\psi)=\lim_j T_{pj}(\psi)$ exists and it defines a bounded functional in $(L^1(m))^\star$.  In particular  there is $A\in L^\infty(m)$ such that 
$$T(\psi)=\int A \psi \ dm$$
for every $\psi\in BV$. We conclude that 
\begin{equation}  \label{bbb}  \int \alpha \ d\psi = - \alpha(a)\psi(a) + \int A \psi \ dm -  \int \psi \Big( \sum_{n=0}^{\infty}  \sum_{i=0}^{p-1}  \phi\circ f^{np+i} \Big) dm\end{equation}
for every $\psi \in BV$.  For $\psi \in C^\infty(I)$ we have
\begin{equation}  - \int \alpha \ D\psi \ dm  = \alpha(a)\psi(a)  -\int A \psi \ dm +  \int \psi \Big( \sum_{n=0}^{\infty}  \sum_{i=0}^{p-1}  \phi\circ f^{np+i} \Big) dm.\end{equation}

Taking $\psi=1_{[a,x]}$ we get
   $$\alpha(x)= \alpha(a)+ q(x)+  \int1_{[a,x]}  \Big( \sum_{n=0}^{\infty}  \sum_{i=0}^{p-1}  \phi\circ f^{np+i} \Big) dm.$$
   where $q(x)=-T(1_{[a,x]})$  is a Lipschitz function , since
   $$|q(y)-q(x)|=|T(1_{[x,y]})|\leq |A|_{L^\infty(m)}  |1_{[x,y]}|_{L^1(m)}=  |A|_{L^\infty(m)} |y-x|.$$
   On the other hand, by Theorem \ref{tt} ($B\implies A$) we have that  there is a piecewise $C^{1+\beta}$ function $w$ such that 
   $$w=\beta(f(x))-Df(x)\beta(x),$$
   where
   $$\beta(x)= G(x)+ \int1_{[a,x]}  \Big( \sum_{n=0}^{\infty}  \sum_{i=0}^{p-1}  \phi\circ f^{np+i} \Big) dm$$
and 
$$G(x)= \int \phi \Big( \sum_{\lambda\in \Lambda_1\setminus\{1\}}  \frac{1}{1-\lambda}  \Phi_\lambda(1_{[a,x]})    \Big) \ dm$$
when $\Lambda_1\setminus\{1\} \neq \emptyset$, or $G(x)=0$  otherwise.  Consequently if $H=\alpha-\beta=\alpha(a)+ q-G$ we have 
   $$v-w= H\circ f -Df \ H.$$
  Since $H$ is a Lipschitz function one can derive  this expression  to obtain
  $$ \frac{Dv-Dw+ D^2f \ H}{Df} =DH\circ f -DH.$$
It  follows that   $$DH\circ f -DH$$ is a piecewise $C^\beta$ function on $I$  and consequently $B.$ holds. 
   
\end{proof}

\begin{lemma} Let $\psi \in L^1(m)$. If  for some $n\in \mathbb{N}^\star$ we have that 
\begin{equation}\label{o} \int \Phi_{1,n}( \gamma)  \sum_{i=0}^{n-1} \psi \circ f^i \ dm=0 \text{ for every }  \gamma\in BV \end{equation}  
then for every  $n\in \mathbb{N}^\star$ we have that  (\ref{o}) holds.
\end{lemma} 
\begin{proof}  We have
$$\Phi_1^n=  \sum_{\beta \in \Lambda_1, \beta^n=1} \Phi_\beta,$$
so
\begin{align*}
&\int \Phi^n_1( \gamma)  \sum_{i=0}^n \psi \circ f^i \ dm\\
&=\int \Big( \sum_{\beta \in \Lambda_1, \beta^n=1} \Phi_\beta (\gamma)\Big) \Big(    \sum_{i=0}^{n-1} \psi \circ f^i \Big) \ dm\\
&=\int  \psi  \sum_{i=0}^{n-1}  \sum_{\beta \in \Lambda_1, \beta^n=1} L^i(\Phi_\beta (\gamma)) \ dm\\
&=\int  \psi   \sum_{\beta \in \Lambda_1, \beta^n=1}  \sum_{i=0}^{n-1} \beta^i \Phi_\beta (\gamma) \ dm\\
&=\int  \psi \Big( n  \Phi_1 (\gamma) + \sum_{\beta \in \Lambda_1, \beta^n=1, \beta\neq 1} \frac{1-\beta^n}{1-\beta} \Phi_\beta (\gamma)\Big)\ dm\\
&=n \int  \psi  \Phi_1 (\gamma)\ dm.
\end{align*}
This completes the proof.
\end{proof}

\section{Deformations}   
A family $f_t\in  \mathcal{B}^k_{exp}(C)$, with $t\in (a,b)$,  is a $C^j$ family if $t\mapsto f_t$ is a $C^j$ function from $(a,b)$ to $ \mathcal{B}^k(C)$.  We say that a  $C^j$ family is {\it Lasota-Yorke stable} if for every compact $K\subset (a,b)$  there is $\Crr{b}$ and $\Crr{c}$  such that 
\begin{equation}\label{lys} |L_t \gamma|_{BV}\leq    \Crr{c} |\gamma|_{BV}+ \Crr{b}|\gamma|_{L^1}. \end{equation}
for every $\gamma \in BV$ and for every $t \in (a,b)$. Here $L_t$ is the Ruelle-Perron-Frobenius operator of $f_t$.    We say that $f\in \mathcal{B}^k_{exp}(C)$ is {\it locally Lasota-Yorke stable} if there is a open neighbourhood of $f$ in $\mathcal{B}^k_{exp}(C)$ such that all maps in there satisfies the Lasota-Yorke inequality on $BV$ with the same constants.

A $C^j$-{\it deformation } of $f \in \mathcal{B}^k_{exp}(C)$ is a $C^j$ family $f_t\in  \mathcal{B}^k_{exp}(C)$ such that $f_t$ is topologically conjugated to $f$ for every $t\in (a,b)$.

Let
$$\hat{C}=\{(c_0,+),(c_n,-)\} \cup (\{c_1,\dots,c_{n-1}\}\times \{+,-\}).$$
The set of critical relations $R_f$ of $f\in  \mathcal{B}^k_{exp}(C)$ is 
$$R_f=\{(x,y,k)\colon \ x,y \in \hat{C}, k\in \mathbb{N}^\star, \ and \ f^k(x)=y   \}.$$

For $f\in  \mathcal{B}^k_{exp}(C)$  let 
$$\mathcal{O}^f=\cup_{i\geq 0} f^i(\hat{C})$$
and $$\ell^\infty(\mathcal{O}^f)=\{ v\colon \mathcal{O}^f\rightarrow \mathbb{R}\colon \sup_{a\in \mathcal{O}^f} |v(a)|< \infty\}.$$
Then we can define the linear  functional $$J(f,x,\cdot)\colon \ell^\infty(\mathcal{O}^f) \rightarrow \mathbb{R},$$  with  $x\in \hat{C}$ as 
$$J(f,x,v)= \sum_{i=0}^{k-1} \frac{v(f^i(x))}{Df^i(f(x))},$$
if $f^i(x)\not\in \hat{C}$ for $1\leq i< k$ and $f^k(x)\in \hat{C}$, and
$$J(f,x,v)= \sum_{i=0}^{\infty} \frac{v(f^i(x))}{Df^i(f(x))},$$
if $f^i(x)\not\in \hat{C}$ for $i\geq 1$. 

In particular $J(f,x,v)$ is well defined for  $v\in \mathcal{B}^k(C)$.

\begin{theorem}[Characterization of infinitesimal deformations] \label{infc} Let $f\in  \mathcal{B}^k_{exp}(C)$ and $v\in \mathcal{B}^k(C)$. The following statements are equivalent.

\begin{itemize} 
\item[A.]   We have $J(f,c,v)=0$ for every $c\in \hat{C}$.

\item[B.]  There is a continuous function $\alpha\colon I \rightarrow \mathbb{R}$ such that 
$$v=\alpha\circ f - Df\cdot \alpha$$
on $I\setminus C$ and $\alpha(c)=0$ for every $c\in C$. 
\item[C.]  There is a continuous function $\alpha\colon I \rightarrow \mathbb{R}$ such that 
\begin{equation} \label{ioio} v=\alpha\circ f - Df\cdot \alpha\end{equation} 
on $I\setminus C$ and $\alpha(c)=0$ for every $c\in C$. Moreover there is $\phi\in C^\beta$ such that 
\begin{equation} \label{rep} \alpha(x)= g(x)+ \int 1_{[c_0,x]}  \sum_{n=0}^{\infty}  \sum_{i=0}^{p-1}  \phi\circ f^{np+i}  \ dm,\end{equation} 
where $p=p(f)$ and $g$ is a Lipchitz function. Indeed we can choose
 $$\phi(x)=\frac{Dv+D^2f\alpha}{Df}.$$
\item[D.]  There is a Log-Lipchitz function  $\alpha\colon I \rightarrow \mathbb{R}$ such that 
$$v=\alpha\circ f - Df\cdot \alpha$$
on $I\setminus C$ and $\alpha(c)=0$ for every $c\in C$. 
\end{itemize} 
Additionally, the constants in the Log-Lipchitz condition on $\alpha$ and in  the Lipchitz  constant of $g$ depends only on the constants of the Lasota-Yorke inequality, $p(f)$, $|f|_{2}$ and $|v|_1$. 
\end{theorem}
\begin{proof}  We already proved that $B$, $C$ and $D$ are equivalent. \\

\noindent  $\it A\implies B$.  if $x\in C$ define $\alpha(x)=0$ and $M(x)=0$.  If $f^k(x)\not\in C$ for every $k\in \mathbb{N}$ let $M(x)=+\infty$.
Otherwise  $x\not\in C$ and  there is $k(x)\geq 1$ such that $f^{M(x)}(x)\in C$ and $f^i(x)\not\in C$ for $0\leq i< M(x)$. Define
$$\alpha(x)=-\sum_{j=0}^{M(x)-1} \frac{v(f^j(x))}{Df^{j+1}(x)}.$$
Since the set of points $Q$  that eventually arrive at $C$ is countable, it is easy to verify that (\ref{ioio}) holds for all points $x$ except maybe those in  $Q$.  So it is enough to show that $\alpha$ is continuous to complete the proof.  One can see that 
$$\lim_{x\rightarrow b^+} \alpha(x)= -\sum_{j=0}^\infty \frac{v(f^j(b^+))}{Df^{j+1}(b^+)}.$$
Let $b\in [-1,1]$. If $f^j(b^+)\not\in \hat{C}$ for every $j$ define $n_0^+=+\infty$, $n_1^+=+\infty$ and $k_0^+=1$. Otherwise let  $n_0^+<n_1^+< n_2^+< \dots$, with $k< k_0^+$, where $k_0^+ \in \mathbb{N}\cup \{+\infty\}$, be the sequence of all times $j$ such that $f^{j}(b^+)\in \hat{C}$. If $k_0^+\in \mathbb{N}$ define $n_{k_0^+}^+=+\infty$.  Note that in all cases $n_0^+=M(b)$. Then
\begin{align*} \sum_{j=0}^\infty \frac{v(f^j(b^+))}{Df^{j+1}(b^+)}&=\sum_{j=0}^{M(x)-1} \frac{v(f^j(b))}{Df^{j+1}(b)}+ \sum_{k=0}^{k_0^+-1}  \sum_{j=n_k^+}^{n_{k+1}^+-1}  \frac{v(f^j(b^+))}{Df^{j+1}(b^+)}\\
&=\sum_{j=0}^{M(x)-1} \frac{v(f^j(b))}{Df^{j+1}(b)}+ \sum_{k=0}^{k_0^+-1} \frac{1}{Df^{n_k^++1}(b^+)} \sum_{j=0}^{n_{k+1}^+-n_k^+-1}  \frac{v(f^j(f^{n_k^+}(b^+)))}{Df^{j}(f(f^{n_k^+}(b^+))}\\
&=\sum_{j=0}^{M(x)-1} \frac{v(f^j(b))}{Df^{j+1}(b)}+ \sum_{k=0}^{k_0^+-1} \frac{1}{Df^{n_k^++1}(b^+)} J(f,f^{n_k^+}(b^+) ,v)\\
&=\sum_{j=0}^{M(x)-1} \frac{v(f^j(b))}{Df^{j+1}(b)}=\alpha(b).
\end{align*} 
Here if no natural number satisfies  the condition in the sum,  consider its value to be zero.

\noindent In an analogous way,  If $f^j(b^-)\not\in \hat{C}$ for every $j$ define $n_0^-=+\infty$, $n_1^-=+\infty$ and $k_0^-=1$. Otherwise $n_0^-<n_1^-< n_2^-< \dots$, with $k< k_0^-$, where $k_0^-\in \mathbb{N}\cup \{+\infty\}$, be the sequence of all times $j$ such that $f^{j}(b^-)\in \hat{C}$. If $k_0^-\in \mathbb{N}$ define $n_{k_0}^-=+\infty$. Again, in all cases we have  $n_0^-=M(b)$.  Then one can similarly conclude that 
$$\lim_{x\rightarrow b^-} \alpha(x)  =-\sum_{j=0}^{M(x)-1} \frac{v(f^j(b))}{Df^{j+1}(b)}=\alpha(b).$$
So we conclude that $\alpha$ is continuous at $b$. Since we concluded that $\alpha$ is continuous, it follows from Theorem \ref{tt} that $B.$ holds. \\

\noindent  $\it D\implies A$. Note that for every $k$ we have

\begin{equation} \label{oo} \sum_{j=0}^{k-1}   Df^{k-j-1}(f^{j+1}(x))\cdot  v(f^j(x))  =\alpha\circ f^k(x) - Df^k(x) \cdot \alpha(x)\end{equation} 
whenever $x$ is not a critical point of $f^k$. 
Let $c\in \hat{C}$ and suppose that there is $c'\in \hat{C}$ and $M\geq 1$  such that $f^M(c)=c'$ and $f^k(c)\not\in \hat{C}$ for every $1\leq k< M$. Since $\alpha(c)=\alpha(c')=0$, when $x$ tends to $c$ in (\ref{oo}) with $k=M$ we obtain
$$ \sum_{j=0}^{M-1}   Df^{M-j-1}(f^{j+1}(c))\cdot  v(f^j(c))  =0.$$
If we divide by $Df^{M-1}(c)$ we get
$$\frac{J(f,c,v)}{Df(c)}=  \sum_{j=0}^{M-1}   \frac{ v(f^j(c)) }{Df^{j+1}(c)}  =0,$$
so $J(f,c,v)=0$. On the other hand if $f^k(c)\not\in \hat{C}$ for every $k$. When we divide   (\ref{oo}) by  $Df^{k-1}(c)$  and $x$ tends to $c$ we get
 $$\sum_{j=0}^{k-1}   \frac{ v(f^j(c)) }{Df^{j+1}(c)}  = \frac{\alpha(f^k(c))}{Df^k(c)} -\alpha(c).$$
Since $\alpha(c)=0$ and $\alpha$ is bounded we get 
 $$\frac{J(f,c,v)}{Df(c)}=\sum_{j=0}^{\infty}   \frac{ v(f^j(c)) }{Df^{j+1}(c)}  = 0.$$
 so $J(f,c,v)=0$. 
\end{proof} 

The representation of $\alpha$ in Theorem \ref{infc}.B may not be unique. However 
\begin{corollary} Let $\alpha$ and  $\phi\in \mathcal{B}^\beta(C)$ be as in Theorem \ref{infc}.C.  Then there is $\gamma \in L^\infty(I)$ such that 
$$\phi -  \frac{Dv-  D^2f\cdot \alpha}{Df}=\gamma\circ f -\gamma.$$
\end{corollary} 
\begin{proof} By Theorem \ref{infc} we have
\begin{equation*}  \tilde{g}(x)+ \int 1_{[c_0,x]}  \sum_{n=0}^{\infty}  \sum_{i=0}^{p-1}  \tilde{\phi} \circ f^{np+i}  \ dm=       \alpha(x)= g(x)+ \int 1_{[c_0,x]}  \sum_{n=0}^{\infty}  \sum_{i=0}^{p-1}  \phi\circ f^{np+i}  \ dm,\end{equation*} 
where $g, \tilde{g}$ are   Lipschitz functions  and
$$\tilde{\phi}= \frac{Dv+  D^2f\cdot \alpha}{Df}.$$
In particular
$$b(x) = \int 1_{[c_0,x]}  \sum_{n=0}^{\infty}  \sum_{i=0}^{p-1}  (\phi-  \tilde{\phi}) \circ f^{np+i}  \ dm$$
is a Lipschitz function. By Theorem \ref{infc} there is a Lipschitz function $\hat{g}$ and $w\in \mathcal{B}^{1+\beta}(C)$ such that the  Lipschitz function $\hat{\alpha} = \hat{g}+ b$ satisfies
$$w=\hat{\alpha}\circ f - Df\cdot \hat{\alpha}.$$
Deriving this expression with respect to $x$ we obtain
$$Dw =    D\hat{\alpha} \circ f \cdot Df   - D^2f \cdot \hat{\alpha} - Df \cdot D\hat{\alpha},$$
so $D\hat{\alpha}\in L^\infty(I)$ satisfies the Livsic cohomological equation
$$\frac{Dw+D^2f\cdot  \hat{\alpha}}{Df} =   D\hat{\alpha} \circ f  -D\hat{\alpha}.$$
\end{proof}

The following theorem  characterizes deformations 
\begin{theorem}[Characterization of deformations]\label{char}  Let  $f_t\in  \mathcal{B}^k_{exp}(C)$, with $t\in (c,d)$, be a  $C^j$ family. The following statements are equivalent.
\begin{itemize}
\item[A.] The set of critical relations $R_{f_t}$ does not depend on $t$.
\item[B.] There exists a family of homeomorphisms $h_t\colon I \rightarrow I$ such that $h_t\circ f_0= f_t\circ h_t$, that is, $f_t$ is a deformation.
\item[C.] There is a family of conjugacies $h_t$ as in $B$, such that 
$$(x,t)\mapsto h_t(x)$$
is a continuous function and for each $x\in I$ we have that $h_t(x)$ is $C^{k-1+Lip}$ on the variable $t$.
\item[D.] We have that $J(f_t,x,\partial_sf_s|_{s=t})=0$ for every $i\leq n$, $x\in \hat{C}$  and $t \in (c,d)$. Moreover the family $f_t$ is Lasota-Yorke stable. 
\item[E.] For every  $t\in (a,b)$ there is a  continuous function $\alpha_t\colon [a,b]\rightarrow \mathbb{R}$ such that 
\begin{equation} \label{tce3} \partial_sf_s|_{s=t}  =\alpha_t\circ f_t  -Df_t\cdot \alpha_t\end{equation} 
on $\hat{I}$  satisfying $\alpha(c_i)=0$ for every $0\leq i\leq n$. Indeed $\alpha_t$ is Log-Lipchitz continuous and the  constants in the Log-Lipchitz condition are such that 
$$\sup_{t\in J} |\alpha_t|_{LL} < \infty$$ for  compact  intervals $J$ of $(c,d)$.  For every $x\in [a,b]$ there is a  the unique solution $h_t(x)$ of the initial value problem 
\begin{equation} \label{edo3} \begin{cases} \dot{h}_t(x)=\alpha_t(h_t(x)) \\ h_0(x)=x.   \end{cases} \end{equation}
then we have 
\begin{itemize} 
\item  $h_t(x)$ is defined for every $(x,t)\in [a,b]\times (c,d)$ and $h([a,b])=[a,b]$.
\item For every $t\in (c,d)$  we have that  $h_t\colon [a,b]\rightarrow [a,b]$  is a homeomorphism,
\item We have $h_t\circ f_0 = f_t\circ h_t$ on $[a,b]$.
\item For every compact interval $J\subset (c,d)$ there is $\Crr{w1}\geq 0$ such that 
\begin{equation} \label{hol1}   |h_t\circ h_{t_0}^{-1}(x)- h_t\circ h_{t_0}^{-1}(x)|\leq e^{1-exp(-\Crr{w1}|t-t_0|)} |x-y|^{exp(-\Crr{w1}|t-t_0|)}.  \end{equation} 
and
\begin{equation} \label{hol2}   |h_{t_0}\circ h_{t}^{-1}(x)- h_{t_0}\circ h_{t}^{-1}(x)|\leq e^{1-exp(-\Crr{w1}|t-t_0|)} |x-y|^{exp(-\Crr{w1}|t-t_0|)}.  \end{equation}
\end{itemize} 
\end{itemize} 
\end{theorem}

\begin{lemma}\label{pre}  For every   $f\in  \mathcal{B}^k_{exp}(C)$ the pre-periodic  points are dense in $[a,b]$. 
\end{lemma} 
\begin{proof}  Let   $\mu$  be an absolutely continuous ergodic probability of  $f$. Then the support of $\mu$  is a finite union of intervals up to a zero measure set (Boyarsky and  G\'{o}ra \cite{bg} ). Del Magno, Dias, Duarte and Gaiv\~ao  \cite{mddg} proved that periodic points are dense on the support of $\mu$.  Furthermore  the  union of the basins of these ergodic measures coincides with $[a,b]$. It is easy  to see that for almost every point $x \in [a,b]$ we have that for every $n \in \mathbb{N}^\star$ there is an open interval $I_n$ such that $f^n$ is a diffeomorphism  on $I_n$ and $x\in I_n$. So for almost every point $x\in [a,b]$ we have that $f^n(x)$ belong to the interior of the support of one of those ergodic measures. So we can approximate $f^n(x)$ by periodic point and consequently we can approximate $x$ by preperiodic points. 
\end{proof} 

\begin{lemma} \label{stable} Let $Q$ be  a compact subset of  $\mathcal{B}^\ell_{exp}(C)$, $j\geq 0$ and $\ell\geq 2$. Let $$C_k(f)=\cup_{i=0}^{k-1} f^{-i}C$$ Define 
$$\lambda_Q=\inf_{f\in Q} \inf_{x\in I} |Df(x)| > 1.$$
Let $k_Q\in \mathbb{N}^\star$ such that $\lambda^{k_Q} > 4$. Suppose 
$$m_Q= \inf_{\substack{ f\in Q\\x,y \in C_{k_Q}(f) \\ (x,y)\cap  C_{k_Q}(f) =\emptyset} } |f^k(x^+)-f^k(y^-)| > 0.$$
Then we can choose   constants in the Lasota-Yorke inequality  for $f\in Q$ that depends only on 
$\lambda_Q, m_Q$ and
$$N_Q=\sup_{f\in Q} |D^2f^k|_{L^\infty}+|Df^k|_{L^\infty}.$$
 Moreover
$$\sup_{f\in Q} p(f)< \infty.$$
\end{lemma}
\begin{proof} It is easy to check in the proof of Lasota-Yorke inequality in Broise \cite{broise} that  constants in the Lasota-Yorke inequality can the chosen  depending only on  $\lambda_Q, m_Q, N_Q$.
To show theupper bound on $p(f)$, suppose that there is $f_n \in Q$ such that $\lim_n p(f_n)=\infty$. That implies that $\Lambda_{f_n}$ has a cyclic group with $p(f_n)$ elements. So given a closed arc $J\subset \mathbb{S}^1$, we have that for $n$ large enough there is $\lambda_n \in J\cap \Lambda_{f_n}$, so there is $v_n\in BV$ with $|v_n|_{L^1(m)}=1$ such that $L_{f_n}v_n=\lambda_n v_n$.  Without loss of generality we can assume $\lim_n f_n=f\in Q$. The uniformity in the constants of  Lasota-Yorke inequality for $f_n$ easily implies that $\sup_n |v_n|_{BV}< \infty$, so we can assume that $\lim_n v_n =v$ in $BV$, with $|v|_{L^1(m)}=1$,  $\lim_n \lambda_n =\lambda \in J$,   and  consequently $L_fv=\lambda v$, so $\lambda\in \Lambda_f\cap J$. Since $J$ is arbitrary, it follows that  $\mathbb{S}^1$ is contained in the spectrum of $f$, which is not possible. So $\sup_{f\in Q} p(f)< \infty$. 
\end{proof} 

 The only  harmless distinction of the  following statement with Theorem 5.2.1 in Chemin \cite{chemin} is that $\alpha_t$ is not defined for every $t\in \mathbb{R}$.
 
\begin{proposition}[Theorem 5.2.1 in Chemin \cite{chemin}] \label{chemin}  Consider the ordinary differential equation 
\begin{equation} \label{ode4}  \begin{cases} \dot{h}_t(x)=\alpha_t(h_t(x)) \\ h_0(x)=x.   \end{cases} \end{equation} 
such that 
\begin{itemize} 
\item The function
$$(x,t)\mapsto \alpha_t(x)$$
is continuous,
\item There are  $\Cll{w1}$ and $\Cll{w2}$  such that  functions $\alpha_t\colon \mathbb{R}\rightarrow \mathbb{R}$  satisfy
$$|\alpha_t(x)-\alpha_t(y)|\leq \Crr{w1}|x-y|(1-\ln |x-y|)$$
for every $x,y$ such that $|x-y|< 1$, and 
$$|\alpha_t(x)|\leq  \Crr{w2}$$
for every $x \in \mathbb{R}$ and $t\in (c,d)$.  
\end{itemize} 
Then (\ref{ode4}) is uniquely integrable and it has a solution $h_t(x)$  defined for every $t\in (c,d)$ and moreover 
$$|h_t(x)-h_t(y)|\leq e^{1-exp(-\Crr{w1}t)} |x-y|^{exp(-\Crr{w1}t)}. $$
provided $|h_a(x)-h_a(y)|< 1$ for every $a\in [0,t]$ and $t\geq 0$.
\end{proposition} 
\begin{proof} Note that there are solutions $h_t(x)$ defined for every $t\in (c,d)$ since $(x,t)\mapsto \alpha_t(x)$ is a continuous and bounded function.  Now apply the same methods of  the proof of Theorem 5.2.1 in Chemin \cite{chemin}.
\end{proof}

\begin{proof}[Proof of Theorem \ref{char}] Due Lemma \ref{pre}  we can use the same argument in the proof of Theorem 1 in  Baladi and S. \cite{smooth} to show that $A$, $B$ and $C$ are equivalent.  Note that $\it E\implies A$ is obvious. \\

\noindent  $\it C\implies D$. The proof  that $C$ implies  $J(f_t,x,\partial_sf_s|_{s=t})=0$ for every $i\leq n$, $c\in \hat{C}$  and $t \in (a,b)$ is also quite similar to \cite[Theorem 1]{smooth}, so we skip it.\\
 It remains to show that   $f_t$ is Lasota-Yorke stable. Indeed, note that  $C$ implies that  if $C_k(t)$ is the set of critical points of $f_t^k$ then $h_t(C_k(0))=C_k(t)$ and consequently  for every compact subset $Q\subset (a,b)$ we have 
$$ \inf_{\substack{ t\in Q\\x,y \in C_k(t) \\ (x,y)\cap  C_k(t) =\emptyset} } |f^k_t(x^+)-f^k_t(y^-)| > 0,$$
so by Lemma \ref{stable} we conclude that  $f_t$ is Lasota-Yorke stable. \\

\noindent  $\it D\implies E$.  By Theorem \ref{infc} there are log-Lipchitzian solutions $\alpha_t$ for  (\ref{tce3}) such that $\alpha_t(c)=0$ for every $c\in \hat{C}$ and $\sup_{t\in J} |\alpha_t|_{log-Lip} < \infty$ and  $\sup_{t\in J} |\alpha_t|_{L^\infty(m)} < \infty$ on any compact interval $J$. We can extend $\alpha_t$ to be zero outside $I$ in such way that $\alpha_t$ became  log-Lipchtizian  on $\mathbb{R}$ 
with a similar uniform bound on the log-Lipchitizian norm and the sup norm.  Note that 
$$(x,t)\mapsto \alpha_t(x)$$ 
is continuous.  By Osgood we have that (\ref{edo3}) has a unique solution $h_t(x)$ defined for every $t\in (a,b)$. Note that  $h_t$ is defined for every $t \in (a,b)$ since $\alpha_t(x)$ are uniformly bounded on compact intervals in $(a,b)$.  Note also that since $h_t(0)=0$ and $h_t(1)=1$ for every $t$ we conclude that $h_t(x)\in (0,1)$ for every $x\in (0,1)$ and 
$t\in (a,b)$. 

Let $x\in [0,1]$. It follows from  (\ref{tce3}) that $h_t(f_0(x))$ and $f_t(h_t(x_0))$ are both solutions of  $\dot{y}=\alpha_t(y)$ with initial condition $y(0)=f_0(x) \in [0,1]$ so the unique integrability implies that $h_t\circ f_0=f_t\circ h_t$ (the verification of this when  solutions cross critical points is a slightly more delicate,  see Baladi and S. \cite{smooth} for details).  

Let $J\subset (a,b)$ be a compact interval and
$$\Crr{w1}=\sup_{t\in J} |\alpha_t|_{LL}.$$
By Proposition \ref{chemin} we have that (\ref{edo3})   defines a flow that satisfies 
$$|h_t(x)-h_t(y)|\leq e^{1-exp(-\Crr{w1}t)} |x-y|^{exp(-\Crr{w1}t)}$$
for $t\geq 0$.  This proves (\ref{hol1}) for $t_0=0$ and $t\geq 0$. For a general $t_0$, note that it is enough to show (\ref{hol1}) and (\ref{hol2}) when $t\geq t_0$.   To show (\ref{hol1}) for $t\geq t_0$  consider the smooth family $\tilde{f}_t=f_{t_0+t}$, apply the same argument to $\tilde{f}_t$ and note that $\tilde{h}_t \circ \tilde{f}_0 = \tilde{f}_t \circ \tilde{h}_t$, where $\tilde{h}_t =h_{t_0+t}\circ h_{t_0}^{-1}$.  Then we obtain
$$|h_{t_0+t}\circ h_{t_0}^{-1}(x)-h_{t_0+t}\circ h_{t_0}^{-1}(y)|\leq e^{1-exp(-\Crr{w1}t)}|x-y|^{exp(-\Crr{w1}t)}$$
for every $t\geq 0$ such that $t_0, t_0+t\in J$.

To show (\ref{hol1}), fix $t\geq t_0$.  One  can apply a similar argument to the family  $\tilde{f}_s=f_{t-s}$ since in this case  $\tilde{h}_s \circ \tilde{f}_0 = \tilde{f}_s\circ \tilde{h}_s$, where $\tilde{h}_s =h_{t-s} \circ h_{t}^{-1}$, and we obtain 
$$|h_{t-s} \circ h_{t}^{-1}(x)-h_{t-s} \circ h_{t}^{-1}(y)|\leq e^{1-exp(-\Crr{w1}s)}|x-y|^{exp(-\Crr{w1}s)}$$
for $s\geq 0$ such that $t, t-s\in J$.  Choosing  $s=(t-t_0)$  we obtain 
$$|h_{t_0} \circ h_{t}^{-1}(x)-h_{t_0} \circ h_{t}^{-1}(y)|\leq e^{1-exp(-\Crr{w1}(t-t_0))}|x-y|^{exp(-\Crr{w1}(t-t_0))}.$$
which implies (\ref{hol2}) for $t\geq t_0$, $t,t_0\in J$.
\end{proof}

 \section{Flexibility of multipliers} 
 
 \begin{proposition}\label{flex}   Let $f\in  \mathcal{B}^{2}_{exp}(C)$. Consider the set indexed family $\mathcal{F}$  containing all functionals of the following types
 \begin{itemize}
 \item[A.] Functionals of the form $\Psi_{c,0}(v)=J(f,c,v)$, with $c \in \hat{C}.$
 \item[B.] Functionals of the form 
 $$\Psi_{\mathcal{O}(p),1}(v)= \sum_{j=0}^{m-1}\frac{Dv(f^j(p))+D^2f(f^j(p))\cdot\alpha_{v,j,p}}{Df(f^j(p))}$$
 where $p\in \hat{I}$ is a periodic point with period $m$ and
 $$\alpha_{v,p,j}= -\sum_{k=j}^\infty \frac{v(f^k(f^j(p)))}{Df^{k+1}(f^j(p))}.$$
 \end{itemize}
 Then $\mathcal{F}$ is a linear independent indexed family  in $\mathcal{B}^{\ell}(C)^\star$, for every $\ell\geq 1$.
 \end{proposition} 
 \begin{proof} It is enough to show the following statement. Given a finite sequence of periodic points $p_1,\dots, p_k$ in {\it distinct} orbits, and $c_1, \dots, c_{2n-2}$ be an  enumeration of the  elements of $\hat{C}$,  the linear transformation
 $$T\colon \mathcal{B}^{\ell}(C) \rightarrow \mathbb{R}^{k+ 2n-2}$$
 given by
 $$T(v)=(\Psi_{\mathcal{O}(p_1),1}(v), \dots, \Psi_{\mathcal{O}(p_k),1}(v), J(f,c_1,v), \dots, J(f,c_1,v))$$
 has $\mathbb{R}^{k+ q}$ as its image. Of course it is enough to show that $T( \mathcal{B}^{\ell}(C))$ is dense in $\mathbb{R}^{k+ q}$. Let $m_i$ be the period of $p_i$. 
 
 Let $u_0$ and $u_1$ be  $C^\infty$ functions defined in  $\mathbb{R}$  with support in $[-1,1]$, such that $u_0(0)=Du_1(0)=1$, $Du_0(0)=u_1(0)=0$ and $|u_i|_\infty\leq 1$ for $i=0,1$.   For every $x_0\in \hat{I}\setminus \hat{C}$ and $\delta > 0$ define 
 $$u(x_0,a,b,\delta,x)= au_0((x-x_0)/\delta)+  b\delta u_1((x-x_0)/\delta).$$
 Note that the support of this function is in $[-\delta,\delta]$, and
 \begin{align*} &u(x_0,a,b,\delta,x_0)=a, \\
                       &D_x u(x_0,a,b,\delta,x_0)=b.
 \end{align*} 
 If $x_0=c^+ \ (c^-)$, with $c\in C$, define $u(x_0,a,b,\delta,x)$ as before for $x \geq 0 \ (x\leq 0)$, and $0$, otherwise. 
 
 \noindent Given $\epsilon> 0$, choose $N$ such that 
\begin{equation}\label{ta} \sum_{i=N}^\infty  \frac{2n-2}{|Df^{i}(x)|} < \epsilon/2\end{equation} 
for every $x\in \hat{I}$.   Let
$$(e_1,\dots, e_k, w_1,\dots, w_{2n-2})\in  \mathbb{R}^{k+ 2n-2}.$$
For $c\in \hat{C}$ define
$$N_c =\min\  \{ N\}\cup \{k\colon f^k(c)\in \hat{C}, \ k\geq 1\}.$$
Consider the set 
$$\Omega= \{f^i(c), \ c\in \hat{C}, \ i\leq N_c   \}\bigcup \cup_{i} \mathcal{O}(p_i)$$
and
$$\delta_0=\min_{\substack{x,y\in \Omega\\ x\neq y}} |x-y|.$$
We are going to define $a_{x_0}, b_{x_0}$ for every $x_0\in \Omega$.
Define $a_{c_i}=w_i$ for every $i\leq 2n-2$.  And for all
$$x_0\in \Omega\setminus \hat{C}$$
define $a_{x_0}=0$. 
For each $x_0\in \mathcal{O}(p_i)$  define 
$$\beta_{x_0}= -\sum_{k=0}^\infty \frac{a_{f^{k}(x_0)}}{Df^{k+1}(x_0)}.$$

\noindent If $x_0\in \mathcal{O}(p_i)$ we define $b_{x_0}=-D^2f(x_0)\beta_{x_0}$ for $x_0\neq p_i$, and  $b_{p_i}=e_iDf(p_i) -D^2f(p_i)\beta_{p_i}.$ For every 
$$x_0\in \Omega\setminus \cup_{i} \mathcal{O}(p_i)$$
define $b_{x_0}=0$.

Given $\delta \in (0,\delta_0/2)$ define
$$v_\delta(x)=\sum_{x_0\in \Omega} u(x_0,a_{x_0},b_{x_0},\delta,x).$$
One can see that $v_\delta(x)=a_{x}$ and $Dv_\delta(x)=b_x$  for every $x\in \Omega$. In particular
$$\alpha_{v_\delta,p_i,j}=\beta_{f^j(p_i)}$$
for every $x\in \mathcal{O}(p_i)$, $i\leq k$, and consequently 
$$\Psi_{\mathcal{O}(p_i),1}(v)= e_i.$$
Since 
$$\lim_{\delta\rightarrow 0^+} |v_\delta|_\infty=1$$
we have that (\ref{ta}) implies that  for $\delta$ small enough
$$ |J(f,c_i,v_\delta)-w_i |< \epsilon$$
for every $i\leq 2n-2$. This shows that $T( \mathcal{B}^{\ell}(C))$ is dense in $\mathbb{R}^{k+ q}$. 
 \end{proof} 
 
  \begin{remark} Suppose that $f_t$ is a smooth family in the topological class of $f_0$ such that $f_t=f_0+tv + o(t)$ and   $p$  is a $f_0$-periodic point. Let $p_t$ be the  smooth continuation of $p$, that is, $t\mapsto p_t$ is smooth and $f_t^m(p_t)=p_t$.  Then
 $$\partial_t f_t^j(p_t)\big|_{t=0}= \alpha_{v,p,j}$$
 and 
 $$\partial_t \log |Df_t^m(p_t)|\big|_{t=0}= \Psi_{\mathcal{O}(p),1}(v).$$
 So Proposition \ref{flex} (and Theorem \ref{top})  tell us that {\it in the topological class} of $f_0$ we can perturb the multipliers of periodic points in the independent way.
 \end{remark} 

\section{Topological classes are Banach manifolds} 

\begin{lemma}\label{discrete}  Let $f\in  \mathcal{B}^1_{exp}(C)$ and $\{w_c\}_{c\in \hat{C}}\subset  \mathcal{B}^k(C)$ be such that $J(f_,w_c,c)=1$ and $J(f_,d,w_c)=0$ for $d\neq c$. Then for every $v\in \ell^\infty(\mathcal{O}^f)$ there is an unique vector $(t_c)_{c\in \hat{C}} \in  \mathbb{R}^{\hat{C}}$ such that 
\begin{equation} \label{sol} J(f,c,(v(a)+\sum_{c\in \hat{C}} t_c w_c(a))_{a\in \mathcal{O}})=0\end{equation} 
for every $c\in \hat{C}$ and there is an unique  $(\alpha(a))_{a\in \mathcal{O}}\in \ell^\infty(\mathcal{O})$ such that  $\alpha(c)=0$ for every $c\in \hat{C}$ and 
\begin{equation}\label{soll}  v(a)+\sum_{c\in \hat{C}} t_c w_c(a) =\alpha(f(a))- Df(a)\alpha(a).\end{equation} 
\end{lemma} 
\begin{proof} The existence of the set $\{w_c\}_{c\in \hat{C}}$ follows from Proposition \ref{flex}.  It is easy to see that the unique solution of (\ref{sol}) is $t_c=-J(f,v,c)$. Let 
$$w = v+ \sum_{c\in \hat{C}} t_c w_c.$$
For every $a\in \mathcal{O}^f$ define $\alpha(c)=0$ for $c\in \hat{C}$, 
$$\alpha(a) = -\sum_{i=0}^\infty \frac{w(f^i(a))}{Df^i(f(a))} $$
if $f^i(a)\not \in \hat{C}$ for every $i\geq 0$, and 
$$\alpha(a) = -\sum_{i=0}^{N-1} \frac{w(f^i(a))}{Df^i(f(a))}$$
if $f^N(a)\in \hat{C}$ and $f^i(a)\not\in \hat{C}$ for every $i< N$.  Then $\alpha$ satisfies (\ref{soll}). If $\tilde{\alpha}$ is also a solution of (\ref{soll}) then $\beta =\alpha - \tilde{\alpha}$ satisfies $\beta(c)=0$ for every $c\in \hat{C}$ and $Df(b)\alpha(b)=\alpha(f(b))$ for every $b\not\in \hat{C}$. It easily  follows that $\beta=0$, so $\tilde{\alpha}=\alpha$. 
\end{proof}

\begin{theorem}[Topological classes are Banach manifolds with finite codimension]\label{top} Let  $f\in  \mathcal{B}^k_{exp}(C)$, with $k\geq 3$, be locally Lasota-Yorke stable. 
Then the topological class  $\mathcal{T}(f)\subset \mathcal{B}^k_{exp}(C)$ is a $C^{\tilde{r}}$ Banach manifold modelled on the horizontal direction $E^h(f)$, where $\tilde{r}=[k/2],$ and with codimension  $2n-2$.
\end{theorem} 
\begin{proof} Our approach is to use the Implicit Function Theorem. To this end, choose $f_0\in \mathcal{T}(f)$. We will show that there is a neighbourhood $\mathcal{W}$ of $f_0$ such that $\mathcal{T}(f)\cap \mathcal{W}$ is defined implicitly  by an equation involving certain function  $G$ we now define. \\

\noindent  Let $\{w_c\}_{c\in \hat{C}}\subset  \mathcal{B}^k(C)$ be such that $J(f_0,w_c,c)=1$ and $J(f_0,d,w_c)=0$ for $d\neq c$.  Let $U\subset E^h(f_0)\times \mathbb{R}^{\hat{C}}$ be an open neighbourhood of $0$ in $E^h_f \times \mathbb{R}^{\hat{C}}$  such that for every 
$$(v,\{\beta_c\}_{c\in \hat{c}})\in U$$
we have that $$f_{(v,\{\beta_c\}_{c\in \hat{c}})}= f_0+ v+ \sum_{c\in \hat{C}} \beta_c w_c\in  \mathcal{B}^k(C)$$ 
satisfies $$\theta=\inf_{(v,\{\beta_c\}_{c\in \hat{C}})\in U}  |Df_{(v,\{\beta_c\}_{c\in \hat{C}})}|_\infty > 1.$$
Moreover note that
$$(v,\{\beta_c\}_{c\in \hat{c}}) \in   E^h(f_0) \times \mathbb{R}^{\hat{C}} \mapsto    f_{(v,\{\beta_c\}_{c\in \hat{c}})} \in  \mathcal{B}^k(C)$$
parametrizes $\mathcal{B}^k(C)$, so  from now on we use $(v,\{\beta_c\}_{c\in \hat{c}})$ to represent $f_{(v,\{\beta_c\}_{c\in \hat{c}})}$ without further notice. 
Let
$$I_i=(c_i,c_{i+1})$$
and $$\hat{I}_i= \{c_i^+,c_{i=1}^-\}\cup \{x^{\pm}\colon \ x\in (c_i,c_{i+1})\}.$$
By  Merrien \cite{merrien}(see also Fefferman \cite{fefferman} for a historical account of related results)    one can find a bounded linear transformation 
$$T_i\colon C^k[c_i,c_{i+1}]\rightarrow C^k(\mathbb{R})$$
that such that $T_i(g)$ is an extension of $g$. Let 
$$g_{(i, v,\{\beta_c\}_{c\in \hat{C}})}\colon \mathbb{R} \rightarrow \mathbb{R}$$ 
be defined by $g_{(i,v,\{\beta_c\}_{c\in \hat{C}})}= T_i(f_{v,\{\beta_c\}_{c\in \hat{C}})}/[c_i,c_{i+1}])$.
$$(v,\{\beta_c\}_{c\in \hat{C}})\mapsto g_{(i,v,\{\beta_c\}_{c\in \hat{C}})}$$
is  Fr\'echet $C^{\infty}$-differentiable (indeed, an affine map) considering the product norm in   $\mathcal{B}^k(C)\times \mathbb{R}^{\hat{C}}$ and the $C^{k}(\mathbb{R})$-norm in its image. Reducing $U$ if necessary, there are  $\delta > 0$ and $\epsilon > 0$ such that  $g_{(i,v,\{\beta_c\}_{c\in \hat{C}})}$ is a diffeormorphism on $[c_{i}-\delta,c_{i+1}+\delta]$ for every $i$ and $(v,\{\beta_c\}_{c\in \hat{C}}) \in U$ and
 $$\theta_1=\inf_{(v,\{\beta_c\}_{c\in \hat{C}})\in U}  \inf_{x\in [c_{i}-\delta,c_{i+1}+\delta]}   |Dg_{(i,v,\{\beta_c\}_{c\in \hat{C}})}(x)|_\infty > 1,$$
 and moreover
 $$\{ y\colon dist(y,[f_0(c_i^+),f_0(c_{i+1}^-)])< \epsilon \} \subset g_{(\hat{I}_i,v,\{\beta_c\}_{c\in \hat{C}})}([c_{i}-\delta,c_{i+1}+\delta]).$$
 for every $i$ and $(v,\{\beta_c\}_{c\in \hat{C}})\in U$. Let
 
$$J_i= \{ y\colon dist(y,[f_0(c_i^+),f_0(c_{i+1}^-)])\leq  \epsilon/2 \}.$$

So we can consider  the ``inverse branches" 
$$\phi_{i,v,\{\beta_c\}_{c\in \hat{C}}}\colon J_i \rightarrow  \mathbb{R}$$ 
defined by  $ g_{(v,\{\beta_c\}_{c\in \hat{C}})}\circ \phi_{i,v,\{\beta_c\}_{c\in \hat{C}}}(x)=x$. 
By   Farkas and Garay \cite{fg1}\cite{fg2}  for every $r<  k$ the map 
$$(v,\{\beta_c\}_{c\in \hat{c}})\mapsto  \phi_{i,v,\{\beta_c\}_{c\in \hat{C}}}$$
is $C^{r}$ Fr\'echet differentiable considering the product norm in   $\mathcal{B}^k(C)\times \mathbb{R}^{\hat{C}}$ and the $C^{k-r}(J_i)$-norm in its image. Let 
$$\mathcal{O}=\cup_{i\geq 0} f_0^i(\hat{C}).$$
Define $s\colon \hat{I}\rightarrow \mathbb{N}$ by $s(a)=i$  if  $a\in \hat{I}_i.$ Consider
$$ \ell^\infty(\mathcal{O})=\{ (x_a)_{a\in\mathcal{O}}\in \mathbb{R}^{\mathcal{O}}\colon \sup_{a\in \mathcal{O}} |x_a|< \infty\}$$
endowed with the supremum norm, its affine subspace
$$ \ell^\infty_C(\mathcal{O})=\{ (x_a)_{a\in\mathcal{O}}\in \ell^\infty(\mathcal{O}) \colon \ x_{c^\pm}=c \ for \ every \ c\in C\}.$$
and its corresponding  tangent space
$$ \ell^\infty_0(\mathcal{O})=\{ (x_a)_{a\in\mathcal{O}}\in \ell^\infty(\mathcal{O}) \colon \ x_{c^\pm}=0 \ for \ every \ c\in C\}.$$
If 
$$P=(p_a)_{a\in \mathcal{O}},$$
where $p_a=a$ for every $a\in \mathcal{O}$, then  $P\in  \ell^\infty_C(\mathcal{O})$ and there is $\eta > 0$ such that if $X= (x_a)_{a\in\mathcal{O}} \in  \ell^\infty_C(\mathcal{O}) $ satisfies 
$$|X-P|_{\ell^\infty(\mathcal{O})}< \eta$$
then $x_{f_0(a)} \in J_{s(a)}$. Consequently if we denote
$$W=\{X\in \ell^\infty_C(\mathcal{O})\colon |X-P|_{\ell^\infty(\mathcal{O})}< \eta\}$$
we can define 
$$G\colon U \times W\rightarrow  \ell^\infty(\mathcal{O})$$
by 
$$G(v,\{\beta_c\}_{c\in \hat{C}},(x_a)_{a \in \mathcal{O}})= (y_a)_{a \in \mathcal{O}},$$
where $$y_a=\phi_{s(a), v,\{\beta_c\}_{c\in \hat{C}}}(x_{f_0(a)}) -x_a,$$for $a\in \mathcal{O}\setminus \hat{C}$, and  $$y_{a^\pm} =\phi_{s(a^\pm),v,\{\beta_c\}_{c\in \hat{C}}}(x_{f_0(a^\pm)})  - a, $$    if $a\in  C$. For $1\leq r < k$ we have that $G$ is $\tilde{r}$-Fr\'echet differentiable, where $\tilde{r}=\min\{k-r,r\}.$ Choosing $r=[k/2]$ we have $\tilde{r}=[k/2]$.
Note that
$$G(0,0,P)=0.$$
For $c\in \hat{C}$ we have
$$\partial_{\beta_c}G(0,0,P) = \Big( - \frac{w_c(a)}{Df_0(a)} \Big)_{a\in \mathcal{O}},$$
and for $b\in \mathcal{O}\setminus \hat{C}$
$$\partial_{x_b}G(0,0,P) = \Big(  y_a^b \Big)_{a\in \mathcal{O}},$$
where 
$$y_{a}^b=\begin{cases}
\frac{1}{Df_0(a)}  &if \ f_0(a)=b \ and \ b\neq a,\\
\frac{1}{Df_0(a)}  &if \ f_0(a)=b \ and \ a\in \hat{C},\\
\frac{1}{Df_0(a)} -1 &if \ a=f_0(a)=b \ \not\in \hat{C}, \\
 -1  &if \ f_0(a)\neq b \ and \ a=b\not\in \hat{C}, \\
0    &  \text{in other cases.}
\end{cases} $$
So the directional derivative of $G$ with respect to the subspace  $\{0\}\times \mathbb{R}^{\hat{C}}\times \ell^\infty_0(\mathcal{O})$ with $0\in E^h(f_0)$,    is the linear transformation
$$\mathcal{D}\colon \mathbb{R}^{\hat{C}}\times \ell^\infty_0(\mathcal{O}) \mapsto \ell^\infty(\mathcal{O})$$
given by 
\begin{align*} \label{inv2} \mathcal{D}((t_c)_{c\in \hat{C}}, (\alpha(b))_{b\in \mathcal{O}})&= \sum_{c\in \hat{C}} \partial_{\beta_c}G(0,0,P)\cdot t_c  + \sum_{b\in \mathcal{O}\setminus \hat{C}} \partial_{x_b}G(0,0,P)\cdot \alpha(b)\\
&=(z_a)_{a\in \mathcal{O}},\end{align*} 
where  for every $(t_c)_{c\in \hat{C}}$ and $(\alpha(b))_{b\in \mathcal{O}}\in \ell^\infty_0(\mathcal{O})$  the  $a$-th component $z_a$ 
is
\begin{equation}\label{inv2} z_a=- \frac{1}{Df_0(a)} \sum_{c\in \hat{C}} t_c  w_c(a)  + \frac{\alpha(f_0(a))}{Df_0(a)} -\alpha(a).\end{equation} 
We claim that  $\mathcal{D}$  is invertible.  To prove that $T$ is a subjective, let $(z_a)_{a\in \mathcal{O}}\in  \ell^\infty(\mathcal{O})$. Then there is an unique vector $(t_c)_{c\in \hat{C}} \in  \mathbb{R}^{\hat{C}}$ such that 
$$J(f_0,c,(Df_0(a)z_a+\sum_{c\in \hat{C}} t_c w_c(a))_{a\in \mathcal{O}})=0$$
for every $c\in \hat{C}$. By Lemma \ref{discrete} there is an unique  $(\alpha(a))_{a\in \mathcal{O}}\in \ell_0^\infty(\mathcal{O})$ such that 
$$Df_0(a)z_a+\sum_{c\in \hat{C}} t_c w_c(a) =\alpha(f_0(a))- Df_0(a)\alpha(a)$$
for every $a\not\in \hat{C}$ and $\alpha(c)=0$ for every $c\in \hat{C}$. So $((t_c)_{c\in \hat{C}}, (\alpha(b))_{b\in \mathcal{O}})$ satisfies (\ref{inv2}). To prove the injectivity of $\mathcal{D}$, suppose that 
$$\mathcal{D}((t_c)_{c\in \hat{C}}, (\alpha(a))_{a\in \mathcal{O}})=\mathcal{D}((q_c)_{c\in \hat{C}}, (\beta(a)))_{a\in \mathcal{O}}).$$
Then
$$ \frac{1}{Df_0(a)} \sum_{c\in \hat{C}} (t_c-q_c)  w_c(a)  = \frac{(\alpha-\beta)(f_0(a))}{Df_0(a)} -(\alpha-\beta)(a).$$
for every $a\in \mathcal{O}$. That implies
$$J(f_0,c,(\sum_{c\in \hat{C}} (t_c-q_c)  w_c(a) )_{a\in \mathcal{O}})=0$$
for every $c\in \hat{C}$ and consequently $t_c=q_c$ for every $c\in \hat{C}$. By Lemma \ref{discrete} we have $\alpha-\beta=0$. \\

\noindent So $\mathcal{D}$ is invertible. By the Implicit Function Theorem there is an open neighbourhood $O$ of $(0,0,P)$ and a $C^{[k/2]}$-function 
$$\theta\colon \mathcal{U}\rightarrow \mathbb{R}^{\hat{C}}\times \ell_C^\infty(\mathcal{O})$$
where $\mathcal{U}$ is an open neigbourhood of $0$ in $E^h(f_0)$ such that 
\begin{align} \label{ift}&\{(v,(\beta_c)_{c\in \hat{C}},(x_a)_{a \in \mathcal{O}})\in O\colon G(v,(\beta_c)_{c\in \hat{C}},(x_a)_{a \in \mathcal{O}})=0\}\\
&=\{(v,\theta(v))\colon \ v\in \mathcal{U}\} \nonumber.\end{align} 

\noindent It remains to show that this subset  is indeed $\mathcal{T}(f)$ near to $f_0$. This is {\it not} obvious since the definition of $G$ involves the extension operators $T_ i$. \\

\noindent We claim that  for every $\epsilon > 0$ there is a neighbourhood $\mathcal{V} \subset \mathcal{U}$ of $0$ in $E^h(f_0)$ such that for each $v\in \mathcal{V}$ there is $(\beta_c(v))_{c\in \hat{C}}$ such that $|(\beta_c(v))_{c\in \hat{C}}|_{\mathbb{R}^{\hat{C}}}< \epsilon$ such that  if $f$ is defined by 
\begin{equation}\label{ff} f_v= f_0+ v+ \sum_{c\in \hat{C}} \beta_c(v) w_c,  \end{equation} 
 and we define recursively $x_{c^\pm}^{f_v}=f_v(c^\pm)$ for $c\in C$ and $x_{f_0(a)}^{f_v}=f_v(x_a^{f_v})$ we have 
\begin{itemize}
\item[-]  $f_v$ belongs in the topological class of $f_0$,
\item[-] $(v, (\beta_c(v))_{c\in \hat{C}},(x_a^{f_v})_{a \in \mathcal{O}})$ belongs to $O$,
\item[-] $G(v, (\beta_c(v))_{c\in \hat{C}},(x_a^{f_v})_{a \in \mathcal{O}})=0$.
\end{itemize} 
The claim and (\ref{ift})  imply  that there is an open neighborhood $\mathcal{W}$ of $f_0$ such that $$\{(v,\theta(v))\colon \ v\in \mathcal{V}\}=\{(v, (\beta_c(v))_{c\in \hat{C}},(x_a^{f_v})_{a \in \mathcal{O}}) \colon v\in   \mathcal{V}  \}=\mathcal{T}(f)\cap \mathcal{W},$$
so $\mathcal{T}(f)\cap \mathcal{W}$ is a Banach manifold modelled on $E^h(f_0)$. Since $E^h(f_0)$ and $E^h(f_0)$ are isomorphic spaces (once there are subspaces with the same finite codimension) this completes the proof of the theorem. \\

\noindent    To prove the claim, note that,  reducing $\mathcal{U}$  if necessary, there is an open  neighbourhood $\mathbb{U}$ of $0$ in $\mathbb{R}^{\hat{C}}$ such that   for every 
$w\in \mathcal{U}$ and $(\beta_c)_{c\in \hat{C}}\in \mathbb{U}$ and $t\in (-1,1)$ there is a unique $u=u(t,(\beta_c)_{c\in \hat{C}}) \in $ such that 
$$u= w+ \sum_{c\in \hat{C}} \zeta_c w_c,$$
  $J(f,u,c)=0$ for every $c\in \hat{C}$, and
$$f =f_0+tw+ \sum_{c\in \hat{C}} \beta_c w_c.$$
So $u$ is a vector field in the open set 
$$\mathbb{A}=\{f_0+tw+ \sum_{c\in \hat{C}} \beta_c w_c, \ with \ t\in (-1,1) \ and \   (\beta_c)_{c\in \hat{C}}\in \mathbb{U}  \}  $$
of  an affine subspace. This vector field  is continuous due to Theorem \ref{infc}, so we can find an $C^1$ integral curve $g_t$, $|t|< \epsilon$,  with $g_0=f_0$. Here we can choose $\epsilon>0$ that does not depends on $v$.   By Theorem \ref{char} we have that $g_t$ belongs to the topological class of $f_0$ for every $t$. Reducing $\mathcal{U}$ and $\mathbb{U}$  again we can assume that $g_t\in O$ for every $v\in \mathcal{U}$, $|t|< \epsilon$.  Notice that 
$g_t = f_0 + t w + \sum_{c\in \hat{C}} \beta^t_c w_c$
for some $ \beta^t_c \in \mathbb{R}$.   
Define $\mathcal{V}= \epsilon/2 \ \mathcal{U}$.  For every $u\in  \mathcal{V}$ choose $w=2u/\epsilon$ and $f=g_{\epsilon/2}$.  If $h$ is the homeomorphism that satisfies $h\circ f_0= f\circ h$ and $h(c)=c$ for every $c\in C$ then  
$$G(v,(\beta_c^{\epsilon/2})_{c\in \hat{C}},(h(a))_{a \in \mathcal{O}})=0$$
since  $h(a)\in (c_{s(a)},c_{s(a)+1})$ for every $a\in \mathcal{O}\setminus \hat{C}$. By Theorem \ref{char} once can choose $\mathcal{U}$ small enough such that
$$|(h(a))_{a \in \mathcal{O}} -P|_{\ell^\infty(\mathcal{O})}< \eta.$$
This proves the claim. 
 \end{proof}

\section{Quasi-symmetric classes} 
One of the main tools in one-dimensional dynamics  is {\it quasisymmetric rigidity}. Often the conjugacy $h$ between two  real-analytic maps  is a  quasisymmetric map; that is, there is $C$ such that 
$$\frac{1}{C}\leq    \frac{|h(x+\delta)-h(x)|}{|h(x)-h(x-\delta)|}\leq C$$
for every $x, x+\delta, x-\delta$ in the phase space.  The proof that a conjugacy is quasisymmetric typically uses  quite sophisticated methods, including complex analytic extensions of the original real dynamics and quasiconformal maps.  In the setting of real-analytic dynamics, a quasisymmetric conjugacy can be extended to a quasiconformal conjugacy between the extended complex dynamics (a method pioneered by Sullivan \cite{sullivan}), which open the doors to use methods as holomorphic motions, Beltrami paths,  and  quasiconformal surgeries.  See Sullivan \cite{sullivan}, de Melo and van Strien \cite{ms}, Lyubich \cite{lyubich2}, Graczyk and \'{S}wiatek \cite{gs}, Kozlovski,  Shen and  van Strien \cite{kss}, and  Clark, van Strien and Trejo \cite{cst}  for more information. 

Piecewise expanding maps, however, have discontinuities (either in the map itself or its derivative) that are an additional difficulty for (and maybe even precludes) the use of complex extension methods. {\it One may wonder if  the topological classes of such maps coincide with  their quasisymmetric classes}. We are going to see that this  is {\it not} always the case. More unexpectedly,  some topological classes are  {\it laminated}  by quasisymmetric classes, which are also  submanifolds of  finite codimension. 

In this section, we have  strong assumptions on the dynamics of the maps. However, we obtain a quite complete description in this setting.

\begin{proposition}[Obstruction for quasisymmetric conjugacies]\label{obs} Let $f \in  \mathcal{B}^k_{exp}(C)$ be a piecewise expanding map.  Suppose  there is $c\in C\setminus \partial I$ such that 

$$  \inf_{x\in \mathcal{O}^+(f,c^\pm)\setminus \hat{C}} dist(x,\hat{C}) > 0$$
and there  is $N_{c^\pm} , M_{c^\pm} \in \mathbb{N}$  such that $$f^{M_{c^+}} (f^{N_{c^+}}(c^+))=f^{N_{c^+}}(c^+)$$ and  $$f^{M_{c^-}}(f^{N_{c^-}}(c^-))~=~f^{N_{c^-}}(c^-).$$
Let $g \in \mathcal{B}^k_{exp}(C)$ be a map such that there is an orientation preserving homeomorphism $h\colon I\rightarrow I$ such that $h\circ f = g \circ h$ and $h(c^\pm)=c^\pm$. If 
$$\frac{\ln |Df^{M_{c^+}}(f^{N_{c^+}}(c^+))|}{\ln |Df^{M_{c^-}}(f^{N_{c^-}}(c^-))|} \neq \frac{\ln |Dg^{M_{c^+}}(g^{N_{c^+}}(c^+))|}{\ln |Dg^{M_{c^-}}(g^{N_{c^-}}(c^-))|} $$
then $h$ is not quasisymmetric. 
\end{proposition} 
\begin{proof} Let $T$ be a  multiple of    $p(f)$, $N_{c^\pm}$  and  $M_{c^\pm}$ and  $F(x)=f^T(x)$, $G(x)=g^T(x)$. We  have $F^2(c^\pm)=F(c^\pm)$, $h\circ F=G\circ h$,  and 
\begin{equation} \label{klkl} \frac{\ln |DF(F(c^+))|}{\ln |DF (F(c^-))| } \neq \frac{\ln |DG(G(c^+))|}{\ln |DG (G(c^-))|}.\end{equation} 
Define 
$$  d=\frac{1}{2} \inf_{x\in \mathcal{O}^+(F,c)\setminus \hat{C}_F} dist(x,\hat{C}_F) > 0.$$

Let $J^1=[c,c+\delta]$ and $J^2=[c-\delta,c]$  and $q_k=q_k(\delta)$, with $k=1,2$ be  the smallest integers such that 
$$    |DF^{q_k}(c_k)||J^k| > \frac{d}{ \Crr{dist}}.$$
Then $F^{q_k}$ is a diffeomorphism on $J^k$ and 
$$|I|\geq |F^{q_k}J^k|\geq \frac{d}{ \Crr{dist}^2}.$$
Since $F$ and $G$ are conjugated by $h$  there is $\Cll{upper1},\Cll{upper2} > 0$ such that 
$$|I|\geq |G^{q_k}h(J^k)|\geq  \Crr{upper1},$$
and (due an usual bounded distortion argument)
$$    |DG^{q_k}(c_k)||h(J^k)| >  \Crr{upper2},$$
where $\Crr{upper1}$, $\Crr{upper2}$ does not depend on $h$. As a consequence
we have 
$$  - \ln|J^k|  - q_k \ln  |DF(F(c_k))| = O(1)$$
so since $|J^1|=|J^2|=\delta$ and $q_k(\delta)\rightarrow  +\infty$ when $\delta$ tends to zero we obtain 
$$\frac{\ln  |DF(F(c^+))|}{ \ln |DF(F(c^-))|}=\lim_{\delta\rightarrow 0^+} \frac{q_1(\delta)}{q_2(\delta)}.$$
But in other hand
$$ - \ln |h(J^k)| - q_k \ln  |DG(G(c_k))|= O(1),$$
so if $h$ is quasisymmetric we have $\ln |h(J^1)| - \ln |h(J^1)|=O(1)$ and 
$$\frac{\ln  |DG(G(c^+))|}{ \ln |DG(G(c^-))|}=\lim_{\delta\rightarrow 0^+} \frac{q_1(\delta)}{q_2(\delta)}.$$
we conclude that 
$$\frac{\ln  |DF(F(c^+))|}{ \ln |DF(F(c^-))|}= \frac{\ln  |DG(G(c^+))|}{ \ln |DG(G(c^-))|}$$
which is impossible due (\ref{klkl}). 
\end{proof} 

\begin{remark} We may wonder if there are others obstructions to  quasisymmetric conjugacy than the one described by  Proposition \ref{obs}. Note that this obstruction does not occur in piecewise expanding unimodal maps. Is the conjugacy always quasisymmetric in this case? 
\end{remark}

We say that  $f \in  \mathcal{B}^k_{exp}(C)$ satisfies the  Assumption FOorMC if \\

\noindent {\bf Finite orbit or Misiurewicz condition (FOorMC).}   We have 
$$ \inf_{c\in C} \inf_{x\in \mathcal{O}^+(f,c^\pm)\setminus \hat{C}} dist(x,\hat{C}) > 0$$
and all $c\in C\setminus \partial I$  one of the following conditions holds
\begin{itemize}
\item {\bf Type I (Finite orbit).} there is $N_{c^\pm} , M_{c^\pm} \in \mathbb{N}$   such that 
$$f^{M_{c^\pm}} (f^{N_{c^\pm}}(c^\pm))=f^{N_{c^\pm}}(c^\pm).$$
\item  {\bf Type II (Misiurewicz and Continuous).}   $f^i$ is  continuous at $c$ for every $i\geq 0$ and there is $N_c=N_{c^+}=N_{c^-}$ such that $f^i(c)\not\in C$ for every $i\geq N_c$. \\
\end{itemize}

\begin{proposition}[Quasisymmetric conjugacies] \label{quasi2} Let $f \in  \mathcal{B}^k_{exp}(C)$ be a piecewise expanding map satisfying Assumption FOorMC. Let $g\in  \mathcal{B}^k_{exp}(C)$ and suppose there is a homeomorphism $h$  such that $h\circ f = g \circ h$ and $h(c^\pm)=c^\pm$. Suppose that  for every Type I critical point $c\in C$ 
$$\frac{\ln |Df^{M_{c^+}}(f^{N_{c^+}}(c^+))|}{\ln |Df^{M_{c^-}}(f^{N_{c^-}}(c^-))|} = \frac{\ln |Dg^{M_{c^+}}(g^{N_{c^+}}(c^+))|}{\ln |Dg^{M_{c^-}}(g^{N_{c^-}}(c^-))|}.$$
Then $h$ is a  quasisymmetric map.
\end{proposition} 
\begin{proof}Let $T$ be a common multiple of  $p(f)$, $p(g)$, $N_{c^\pm}$  and  $M_{c^\pm}$ for every $c\in C$. Let $F(x)=f^T(x)$. Then  $F\in \mathcal{B}^k_{exp}(C_F)$ and $G\in \mathcal{B}^k_{exp}(h(C_F))$, for a  finite set  $C_F$, and for every $c\in C_F$
\begin{itemize}
\item[-]{\it Type I.} either $F^2(c^\pm)=F(c^\pm)$ and 
\begin{equation}\label{erre} \frac{\ln |DF (F(c^+))| }{\ln |DF (F(c^-))| } = \frac{\ln |DG (G(h(c^+)))| }{\ln |DG (G(h(c^-)))| },\end{equation} 
\item[-]{\it Type II.} or  $F^i$ is  continuous at $c$ for every $i\geq 0$ and   $F^i(c)\not\in C_F$ for every $i\geq 1$. 
\end{itemize} 
Furthermore $F\circ h = h\circ G$,
$$  d_F=\frac{1}{2} \inf_{c\in \hat{C}_F} \inf_{x\in \mathcal{O}^+(F,c)\setminus \hat{C}_F} dist(x,\hat{C}_F) > 0$$
and
$$  d_G=\frac{1}{2} \inf_{c\in \hat{C}_G} \inf_{x\in \mathcal{O}^+(G,c)\setminus \hat{C}_G} dist(x,\hat{C}_G) > 0.$$
Let $x\in I$ and $\delta > 0$ be such that $[x-\delta,x+\delta]\subset I$. We may assume $|\delta|< d_F$.  

Note that there is $\Cll{dist1} >0$ and $\Cll{distp1} > 0$  such that  for every interval $R\subset I$ and $n\in \mathbb{N}$  where $F^i(R)\cap  C_F=\emptyset$  for $i< n$ we have 
\begin{equation}\label{ddis1}   \frac{1}{ \Crr{dist}} \leq    \frac{DF^n(x)}{DF^n(y)}\leq \Crr{dist1}\end{equation} 

\begin{equation}\label{ddis21}  |\ln  |DF^n(x)|  - \ln |DF^n(y)| |\leq \Crr{distp1}|F^n(x)-F^n(y)|.\end{equation} 
for all $x,y\in R$. Moreover $G$ has analogous properties.  Let $q$ be the smallest integer satisfying $$F^q((x-\delta,x+\delta))\cap C_F\not=\emptyset.$$
We have
$$|F^i[x-\delta,x+\delta]|\leq (\min_{z\in \hat{I}} |DF(z)|)^{-i}$$
Define $J=F^q([x-\delta,x+\delta])$.
Let $c\in C_F$ be defined by 
$$\{c\}= F^q((x-\delta,x+\delta))\cap C_F.$$
Denote by $J^1$ and $J^2$ the right and left connected components of $F^q([x-\delta,x+\delta])\setminus \{c\}$. Due 
(\ref{ddis1})  there is $\Cll{ui} > 1$ such that

$$\frac{1}{\Crr{ui}}\leq  \frac{|J^1|}{|J^2|} \leq \Crr{ui}.$$
Fix $y\in [x-\delta,x+\delta]$ such that $F^q(y)=c$.  It follows from (\ref{ddis1})  that  $$ \frac{1}{|DF^q(y)|} \leq 2\Crr{dist} \frac{|h|}{|J|}.$$
Define $c_1=c^+$ and  $c_2=c^-$. 

Given an interval $S=[z_1,z_2]\subset I$ such that $z_k$ is between  $z_{3-k}$ and $c_k$, and 
$$dist(S,c_k)\leq  \Crr{dist1} |S|,$$
let  $q(S)$  be  the smallest integer satisfying 
$$    |DF^{q(S)}(c_k)||S| > \frac{d_F}{ \Crr{dist1}(\Crr{dist1}+1)}.$$
This implies
\begin{equation}\label{order1} \ln |DF^{q(S)}(c_k)| + \ln |S| =O(1).\end{equation} 
Moreover  $F^{q(S)}$ is a diffeomorphism on $[z_{3-k},c_k]$ and 
$$|F^{q(S)}S|\geq \frac{d_F}{ \Crr{dist1}^2(\Crr{dist1}+1)},$$
that implies that $G^{q(S)}$ is a diffeomorphism on $h([z_{3-k},c_k])$ and  there is $\Cll{lowerb}> 0$ such that 
$$|G^{q(S)}h(S)|>  \Crr{lowerb}$$
and
$$    |DG^{q(S)}(c_k)||h(S)| >  \Crr{lowerb},$$
so
\begin{equation}\label{order2}  \ln |DG^{q(S)}(c_k)| + \ln |h(S)| =O(1).\end{equation} 
 Note  that 
\begin{equation}\label{con}  |F^i(S)|\leq ( \max_{z\in \hat{I}}  |DF(z)|)^{-i}. \end{equation} 
for every $i\leq q(S)$.

Let $J=[y_1,y_2]$. If $Q_1$ and $Q_2$ are  the right and left   connected components of $J\setminus \{F^q(x)\}$  then
\begin{equation}\label{lll} \frac{1}{\Crr{dist1}}   \leq  \frac{|Q_1|}{|Q_2|}\leq \Crr{dist1},\end{equation} 
so
$$\frac{1}{1+\Crr{dist1}}\leq \frac{|Q_i|}{|J|}\leq  \frac{\Crr{dist1}}{1+\Crr{dist1}}.$$
Suppose  $c\in Q_1$ (the case $c\in Q_2$ is analogous).  Then $J^1\subset Q_1$  and $Q_2\subset J^2$ 
and
$$dist(z_2,Q_2)\leq \Crr{dist1}|Q_2|,$$

We consider two cases. \\

\noindent {\it First case. $c$ is a type I point.}  Then  (\ref{order1}) and (\ref{order2}) imply 
$$  q(S)= - \frac{\ln |S|}{\ln  |DF(F(c_k))|} + O(1)= - \frac{\ln |h(S)|}{\ln  |DG(G(c_k))|} + O(1) .$$
so there is a $\Cll{outra}$ such that 
$$  \frac{1}{\Crr{outra}} |S|^{\frac{\ln  |DG(G(c_k))|}{\ln  |DF(F(c_k))|}}  \leq    |h(S)|\leq   \Crr{outra} |S|^{\frac{\ln  |DG(G(c_k))|}{\ln  |DF(F(c_k))|}}.$$
By (\ref{erre}) we can define
$$r=\frac{\ln  |DG(G(c^+))|}{\ln  |DF(F(c^+))|}=\frac{\ln  |DG(G(c^-))|}{\ln  |DF(F(c^-))|},$$
and consequently 
\begin{equation} \label{p3}  \frac{1}{\Crr{outra}} |S|^{r}  \leq    |h(S)|\leq   \Crr{outra} |S|^{r}\end{equation} 
for every interval $S$ satisfying the conditions we imposed on $S$. 

So  (\ref{p3}) implies 
 \begin{align*}  \frac{1}{   \Crr{outra}  } |Q_2|^r \leq  |h(Q_2)|\leq   \Crr{outra}  |Q_2|^r.  \end{align*} 
 Since $Q_1= [y_1,c] \cup [c,F^q(x)]$ by  (\ref{p3}) we have
\begin{align*}    |h(Q_1)|&= |h(J^1)|+ |h([c,F^q(x)])| \\
&\leq  \Crr{outra} (|J^1|^r +  |[c,F^q(x)]|^r) \leq 2\Crr{outra}  \max \{ |J^1|,  |[c,F^q(x)]|\}^r\\
&\leq 2\Crr{outra}|Q^1|^r.
\end{align*} 
and
\begin{align*}    |h(Q_1)|&= |h(J^1)|+ |h([c,F^q(x)])| \\
&\geq \frac{1}{ \Crr{outra}} (|J^1|^r +  |[c,F^q(x)]|^r) \geq \frac{1}{ \Crr{outra}}  \max \{ |J^1|,  |[c,F^q(x)]|\}^r\\
&\leq  \frac{1}{ 2^r\Crr{outra}} |Q^1|^r.
\end{align*} 
so there is $\Cll{qs}\geq  1$ such that 
\begin{equation}\label{c1} \frac{1}{\Crr{qs}}\leq \frac{|h(Q_1)|}{|h(Q_2)|}\leq \Crr{qs},\end{equation}
and the bounded distortion of $G$ implies
\begin{equation}\label{c2} \frac{1}{ \Crr{dist1} \Crr{qs}}\leq \frac{|h([x,x+h])|}{|h([x-h,x])|}\leq \Crr{dist1} \Crr{qs}.\end{equation} 
\ \\

\noindent {\it Second case. $c$ is a type II point.}    In this case
$$\ln |DF^{q}(c^+)|-  \ln |DF^{q}(c^-)|=O(1)$$
and
$$\ln |DG^{q}(c^+)|-  \ln |DG^{q}(c^-)|=O(1)$$
for every $q$.   So
$$\ln |DF^{q(S)}(c_1)|+ \ln |S| = O(1),$$
$$\ln |DG^{q(S)}(c_1)|+ \ln |h(S)| = O(1).$$
For $S=Q_2$ we obtain
\begin{align} \label{lll3} &\ln |DF^{q(Q_2)}(c_1)|+ \ln |Q_2| = O(1),\\
&\ln |DG^{q(Q_2)}(c_1)|+ \ln |h(Q_2)| = O(1).\nonumber \end{align}
Note if $\tilde{S}=[s_1,c]\cup [c,s_2]$ then  $2|[s_i,c]|\geq |\tilde{S}|$, for some $i\in \{1,2\}$  we have  $$\ln |\tilde{S}|= \ln |[s_{i},c]|+ O(1)$$
$$\ln |\tilde{S}|>  \ln |[s_{3-i},c]| + \ln 2,$$
so
$$\ln |DF^{q([s_i,c])}(c_1)|+ \ln |\tilde{S}| = O(1),$$
and there is $\Cll{lowerr}$ such that 
$$\ln |DF^{q([s_{3-i},c])}(c_1)|+ \ln |\tilde{S}| >   \Crr{lowerr}.$$
Since $F$ is uniformly expanding we conclude that 
$$\ln |\tilde{S}| + \ln  |DF^{\min \{ q([s_i,c]),q([s_{3-i},c])\}  }(c_1)| = O(1).$$
We can use a similar argument with $G$ and $h(\tilde{S})$ and obtain
$$\ln |h(\tilde{S})| + \ln  |DG^{\min \{ q([s_i,c]),q([s_{3-i},c])\}  }(c_1)| = O(1).$$
Take $\tilde{S}=Q_1$. Then
\begin{equation}\label{lll2} \ln |Q_1| + \ln  |DF^{\min \{ q(J^1),q([c,F^q(x))\}  }(c_1)| = O(1).\end{equation} 
We can use a similar argument with $G$ and $h(\tilde{S})$ and obtain
\begin{equation}\label{lll4} \ln |h(Q_1)| + \ln  |DG^{\min \{ q(J^1),q([c,F^q(x))\}  }(c_1)| = O(1).\end{equation}
Since (\ref{lll}), (\ref{lll2}) and (\ref{lll3})  imply
$$\ln |DG^{q(Q_2)}(c_1)| - \ln  |DF^{\min \{ q(J^1),q([c,F^q(x))\}  }(c_1)|=O(1)$$
so the uniform expansion of $F$ gives us
$$q(Q_2)- \min \{ q(J^1),q([c,F^q(x))\}  =O(1),$$
and finally (\ref{lll3}) and (\ref{lll4}) imply (\ref{c1}) and (\ref{c2}). This completes the proof.
\end{proof}

\begin{theorem}[Quasisymmetric deformations]\label{quasi3}  Let $f_0 \in  \mathcal{B}^k_{exp}(C)$ be a piecewise expanding map satisfying Assumption FOorMC. Let $f_t\in \mathcal{B}^k_{exp}(C)$ be a smooth family, $t\in (c,d)$. The following statements are equivalent
\begin{itemize}
\item[A.] For every $t$ there is a quasisymmetric map $h_t$ such that $h_t(c)=c$ for every $c\in C$ and $ f_t\circ h_t  = h_t\circ f_0$. 
\item[B.] For every Type I critical point $c\in C\setminus \partial I$ and every $t$ we have 
 \begin{align*} 
\label{no} &\frac{1}{\ln |Df_t^{M_{c^+}} (f^{N_{c^+}}(c^+))| } \sum_{i=N_{c^+}}^{N_{c^+} + M_{c^+}-1}  \phi_t(f_t^i(c^+)) \numberthis\\
&= \frac{1}{\ln |Df_t^{M_{c^-}} (f^{N_{c^-}}(c^-))| } \sum_{i=N_{c^-}}^{N_{c^-} + M_{c^-}-1}  \phi_t(f_t^i(c^-)),\end{align*} 
where
$$\phi_t = \frac{Dv_t + D^2f_t\cdot \alpha_t}{Df_t},$$
$v_t=\partial_t f_t$ and $\alpha_t$ is the unique continuous solution of 
\begin{equation}\label{tcer} v_t =\alpha_t\circ f_t - Df_t\circ \alpha_t.\end{equation} 
Moreover the family $f_t$ is Lasota-Yorke stable. 
\item[C.]  For every $t_0$ and $x, x+\delta, x-\delta\in I$ we have that 
\begin{equation} \label{qs} \frac{|h_{t}\circ h_{t_0}^{-1}(x+\delta)-h_{t}\circ h_{t_0}^{-1}(x) |}{|h_{t}\circ h_{t_0}^{-1}(x-\delta)-h_{t}\circ h_{t_0}^{-1}(x) |}\leq (1+O(|t-t_0|) \end{equation} 
\end{itemize} 
\end{theorem} 
\begin{proof} Of course $C\implies A$. \\
\noindent {\it $A\implies B$.}  Since $f_t\in \mathcal{T}(f_0)$ for every $t$,  Theorem \ref{char} implies
that (\ref{tcer})  as a unique continuous solution $\alpha_t$. Proposition \ref{obs} and $A$ implies
$$\frac{\ln |Df_t^{M_{c^+}}(h_t(f_0^{N_{c^+}}(c^+)))|}{\ln |Df_t^{M_{c^-}}(h_t(f_0^{N_{c^-}}(c^-)))|} =\frac{\ln |Df_0^{M_{c^+}}(f_0^{N_{c^+}}(c^+))|}{\ln |Df_0^{M_{c^-}}(f_0^{N_{c^-}}(c^-))|} $$
for every $t$ and Type II critical point $c$.  Deriving with respect to $t$ we obtain
\begin{align*}&\frac{1}{ \ln |Df_t^{M_{c^+}}(h_t(f_0^{N_{c^+}}(c^+)))|} \partial_t \ln |Df_t^{M_{c^+}}(h_t(f_0^{N_{c^+}}(c^+)))| \\&= \frac{1}{\ln |Df_t^{M_{c^-}}(h_t(f_0^{N_{c^-}}(c^-)))|} \partial_t \ln |Df_t^{M_{c^-}}(h_t(f_0^{N_{c^-}}(c^-)))|.\end{align*} 
Since for every $M$ and $N$  and $c\in C\setminus \partial I$ 
\begin{align*} &\partial_t \ln |Df_t^M(h_t(f_0^N(c^\pm)))=\partial_t \Big(  \sum_{i=0}^{M-1}  \ln |Df_t(h_t(f_0^{N+i}(c^\pm)))|  \Big)\\
&= \sum_{i=0}^{M-1} \frac{ \ln |v_t(h_t(f_0^{N+i}(c^\pm)))+Df^2_t(h_t(f_0^{N+i}(c^\pm)))\cdot \alpha_t(h_t(f_0^{N+i}(c^\pm)))}{Df_t(h_t(f_0^{N+i}(c^\pm)))}\\
&=  \sum_{i=0}^{M-1} \phi_t(h_t(f_0^{N+i}(c^\pm)))= \sum_{i=N}^{N+M-1} \phi_t(f_t^{i}(c^\pm)),\end{align*} 
we have that (\ref{no}) holds. The family $f_t$ is Lasota-Yorke stable due Theorem \ref{char}.\\
\noindent {\it $B\implies C$.}  $B$ and Theorem \ref{char} imply that  for every compact interval $K$ we have that $\alpha_t$ are uniformly Log-Lipschitz for $t\in K$. In particular $\alpha_t$ are uniformly $\beta$-H\"older for every $\beta\in (0,1)$,  for $t\in K$. In particular $\phi_t$ are uniformly piecewise $\beta$-H\"older for every $\beta\in (0,1)$,  for $t\in K$. Theorem \ref{tt} implies that 
$$\alpha_t(x)= H_t(x)+ G_t(x)+ \int1_{[a,x]}  \Big( \sum_{n=0}^{\infty}  \sum_{i=0}^{p-1}  \phi_t\circ f^{np+i} \Big) dm,$$
where $H_t$ and $G_t$ are uniformly Lipschitz functions for $t\in K$. G.R. and S. \cite[ Theorem \ref{sum-vvv}]{um} implies that there is $\Cll{uu}$ such that 
$$|\alpha_t(x+\delta)+\alpha_t(x-\delta)-2\alpha_t(x)|\leq \Crr{uu} |\delta|$$
for every $t\in K$, $x,x+\delta,x-\delta\in I$. Since due Theorem \ref{char} we have that $h_t$ are the solutions of the differential equations $\dot{h}_t =\alpha_t\circ h_t$  with initial condition $h_0(x)=x$, we have by Reimann \cite[Proposition 8]{reimann} that $h_t$ is $e^{ \Crr{uu} |t|}$-quasisymmetric and (\ref{qs}) holds for $t_0=0$. The case for general $t_0$ follows for an  argument similar to those used in the the proof of Theorem \ref{char}. 
\end{proof} 

Let $f\in   \mathcal{B}^k_{exp}(C)$ satisfying assumption FOorMC. Define 

$$\Omega_f= \{ \mathcal{O}^+(c^{\pm}), \  \text {$c$ is a Type I critical point}   \}.$$

Let $D_f$ be the dimension of the linear space
$$\{  f\colon \Omega \rightarrow  \mathbb{R}\colon  \  f(\mathcal{O}^+(c^+))= f(\mathcal{O}^+(c^-)), \  \text {for all $c\in C$ that  is a Type I critical point}     \}. $$

Note that $D_f$ is a topological invariant.

\begin{theorem}[Lamination by quasisymmetric classes] \label{lami} Let $f \in  \mathcal{B}^k_{exp}(C)$ be a map satisfying assumption FOorMC. Then 
\begin{itemize}
\item[A.] Given $g_0$ in the topological class $\mathcal{T}$ of $f$ the  quasisymmetric class of $g$ is  an embedded submanifold $M_g$ of codimension $D_f$ in  $\mathcal{T}$. 
\item[B.] Moreover $v \in  \mathcal{B}^k(C)$ belongs to the tangent space of $M_g$ at $g$ if and only if 
 \begin{align*} 
&\frac{1}{\ln |Dg^{M_{c^+}} (g^{N_{c^+}}(c^+))| } \sum_{i=N_{c^+}}^{N_{c^+} + M_{c^+}-1}  \phi(g^i(c^+))\\
&= \frac{1}{\ln |Dg^{M_{c^-}} (g^{N_{c^-}}(c^-))| } \sum_{i=N_{c^-}}^{N_{c^-} + M_{c^-}-1}  \phi(g^i(c^-)) \end{align*} 
for every $c\in C$ of Type I and $\mathcal{O}^+(f,f^{N_{c^+}}(c^+))\neq \mathcal{O}^+(f,f^{N_{c^-}}(c^-))$. Here
$$\phi=-\frac{Dv+D^2f\ \alpha}{Dg}$$
and $\alpha$ is the only continuous solution of $v = \alpha\circ g - Dg \ \alpha$.
\end{itemize}
\end{theorem} 

\begin{proof} By Theorem \ref{top} we have that in a neighborhood $U$ to $f$ its topological class  $\mathcal{T}\cap U$ is a $C^{[k/2]}$ Banach manifold with codimension  $2n-2$ modelled over the Banach space  $E^h(f)$. Given $g_0 \in \mathcal{T}\cap U$ let $M_{g_0}$ be the set of all $g\in \mathcal{T}\cap U$ such that 
$$\frac{\ln |Dg^{M_{c^+}}(g^{N_{c^+}}(c^+))|}{\ln |Dg^{M_{c^-}}(g^{N_{c^-}}(c^-))|} = \frac{\ln |Dg_0^{M_{c^+}}(g_0^{N_{c^+}}(c^+))|}{\ln |Dg_0^{M_{c^-}}(g_0^{N_{c^-}}(c^-))|}$$
for every Type I critical point  $c\in C\setminus \partial I$. Let $e^{Y_c}$ be the right hand side of this expression. By Theorem \ref{obs} and Proposition \ref{quasi2} we have that $g\in \mathcal{T}\cap U$ is conjugate to $g_0$ by a quasisymmetric map if and only if $g\in M_{g_0}$. 
Consider the function 
$$Q\colon  \mathcal{T}\cap U  \rightarrow \mathbb{R}^{\Omega_{g_0}}$$
defined by 
$$Q(g)(\mathcal{O}^+_{g_0}(x))= \ln \ln |Dg^n(h_g(x))|,$$
where $\mathcal{O}^+_{g_0}(x) \in \Omega_{g_0}$ is an $g_0$-orbit with period $n$. 
Here $h_g$ is the unique homeomorphism such that $h_g \circ g_0 = g \circ h_g$.   Note that  $Q$ is a $C^{[k/2]}$ function.  Consider the affine  subspace $S\subset  \mathbb{R}^{\Omega_{g_0}}$ given by
$$\{ (s_{\mathcal{O}^+_{g_0}(x)})_{\mathcal{O}^+_{g_0}(x)\in \Omega_{g_0} }\colon   s_{\mathcal{O}^+_{g_0}(g^{N_{c^+}}(c^+))}- s_{\mathcal{O}^+_{g_0}(g_0^{N_{c^-}}(c^-))}=Y_c \text{ for all $c$ that is  Type I} \}.$$
Note that the tangent space of $S$ is
$$\{ (s_{\mathcal{O}^+_{g_0}(x)})_{\mathcal{O}^+_{g_0}(x)\in \Omega_{g_0} }\colon   s_{\mathcal{O}^+_{g_0}(g^{N_{c^+}}(c^+))}= s_{\mathcal{O}^+_{g_0}(g_0^{N_{c^-}}(c^-))} \text{ for all $c$ that is  Type I} \},$$
which has dimension $D_f$. We have
 $$M_{g_0}=Q^{-1}S.$$
We will apply Submersion Theorem  to  prove that $M_{g_0}$ is a submanifold of codimension $D_f$. If we derive $Q$ at $g$ in the direction $E^h_{g}$ we obtain
$$D_gQ\cdot v = \Big( \frac{\sum_{i=0}^{p(x)-1}  \phi(g^i(h_g(x)))}{\ln |Dg^{p(x)} (h_g(x))| }   \Big)_{\mathcal{O}^+_{g_0}(x)\in \Omega_{g_0}},$$
where $p(x)$ is the period of the $g_0$-orbit of $x$ and 
$$\phi =\frac{Dv+D^2g_0 \ \alpha}{Dg_0},$$
and $\alpha$ is the solution of the equation $v=\alpha\circ g_0 - Dg_0\circ \alpha$.
We need to show that the image of $$D_{g_0}Q\colon T_{g_0}\mathcal{T}\rightarrow \mathbb{R}^{\Omega_{g_0}}$$ is  $\mathbb{R}^{\Omega_{g_0}}$. This follows from Proposition \ref{flex}. The description of the tangent space of $M_{g_0}$ follows from Theorem \ref{quasi3}.
\end{proof}

\section{Relation with  partially hyperbolic framework and it nightmares} 

We can interpret most  of the results wherein the framework  of $2$-dimensional piecewise smooth partially hyperbolic endomorphisms. 

\begin{proposition} \label{ph}Let $f_t \in \mathcal{B}_{exp}(C)$, $t\in [0,1]$, be a smooth family, with $t\in [0,1]$. Define
$$F\colon I  \times [0,1]\rightarrow  I  \times [0,1].$$
as $F(x,t)=(f_t(x),t)$. Then $F$ is a partially hyperbolic piecewise smooth endomorphism in the following sense
\begin{itemize}
\item[A.] The unstable manifolds are horizontal lines. More precisely 
$$W^u(x,t)=I\times \{t\}, \ E^u_{(x,t)}=\mathbb{R}\times \{0\},$$
$W^u(x,t)$ is  invariant and there is $\theta > 1$ such that 
$$|D_x F(x,t)|=|Df_t(x)|\geq \theta$$
for all  $(x,t)$ where $D f_t(x)$ is defined. 
\item[B.] There exists a measurable subset $S$ such  that
$$S^c\cap (I\times\{t\})$$
is countable for every $t$, and a continuous and bounded function
$$\hat{\alpha}\colon S\rightarrow  \mathbb{R}^2$$ 
such that  $E^c_{(x,t)}=< \hat{\alpha}>$. Indeed
$$D F(x,t)\cdot \hat{\alpha}(x,t)= \hat{\alpha}(F(x,t)).$$
\end{itemize} 
\end{proposition} 
\begin{proof} Statement $A$ is obvious. Let
$$S=\{(x,t)\colon f_t^i(x) \text{ is well-defined and} f^i(x)\not\in C \text{ for every $i\geq 0$}\}.$$
Let 
$$\alpha_t(x)= -\sum_{i=0}^\infty  \frac{v_t(f_t^i(x))}{Df_t^{i+1}(x)}$$
for $(x,t)\in S$, where $v_t=\partial_t f_t$.  It is easy to see that 
$$(x,t)\mapsto \alpha_t(x)$$
is continuous and bounded function.  Note that 
$$v_t(x) =\alpha_t(f_t(x))-Df_t(x)\alpha_t(x)$$
for $(x,t)\in S$. Define
$$\hat{\alpha}(x)= \begin{bmatrix}
\alpha_t(x)\\
1
\end{bmatrix}.$$
We have
$$DF(x,t)\cdot \hat{\alpha}(x,t)= \begin{bmatrix}
Df_t(x) & v_t(x)  \\
0 & 1
\end{bmatrix}\cdot \begin{bmatrix}
\alpha_t(f_t(x))\\
1
\end{bmatrix}=  \begin{bmatrix}
\alpha_t(x)\\
1
\end{bmatrix}=  \hat{\alpha}(x,t).$$
\end{proof} 

\begin{proposition} Let $F$ be as in Proposition \ref{ph}, and assume that  $f_t$ is Lasota-Yorke stable. Then $E^c$ extends to a continuous distribution on $I\times [0,1]$ if and only if $f_t$ is a deformation, that is, $f_t$ is topologically conjugate with $f_0$ for every $t$. 
\end{proposition} 
\begin{proof} $E^c$ has as continuous extension to  $I\times [0,1]$ if and only if $\alpha_t$ has a continuous extension $\alpha_t\colon I\rightarrow I$ for every $t$. But it is equivalent  to $f_t$ be a deformation by the characterization of deformations by Theorem \ref{char}.
\end{proof} 

\begin{proposition} \label{clt} Let $F$ be as in Proposition \ref{ph}, and assume that  $f_t$ is a deformation. Then for every $t$ the function 
$$x\mapsto \hat{\alpha}(x,t)=(\alpha_t(x),1)$$
is Log-Lipschitz continuous. The central direction is uniquely integrable and  the holonomy between unstable manifolds  trough the central  lamination  are H\"older.  Moreover there are examples such that for every $t$ the function
$$x\mapsto \alpha_t(x)$$
satisfies a Central Limit Theorem for its modulus of continuity. 
\end{proposition} 
\begin{proof} This follows from Theorems \ref{infc} and \ref{char} and   G.R. and  S. \cite[Theorem \ref{sum-cite}]{um}.
\end{proof} 

\begin{proposition} Choose a deformation $f_t$ such that for every support  of an ergodic absolutely continuous  $f_0$-invariant probability $\mu_0$  there is a $f_0$-periodic point $q$, $f_0^M(q)=q$  in the support  $\mu_0$ such that 
$$t\mapsto |Df_t^M(h_t(q))|$$
is injective where $h_t$ is the conjugacy  between $f_0$ and $f_t$. Then the center foliation $W^c$ of $F$ is a Fubini's nightmare.  Indeed, there is a subset $A$ of positive Lebesgue measure in $I\times [0,1]$ such that $A\cap I\times \{t\}$ is an atomic set for every $t\in [0,1]$. 
\end{proposition} 
\begin{proof} This is quite  similar to the Katok's example (see Milnor \cite{milnorp}). Let $A$ defined in the following way. A point $(x,t)$ belongs to $A$ if and only if $x$ is typical with respect to some ergodic absolutely continuous invariant probability $\mu_{x,t}$ of $f_t$.  $A$ has full Lebesgue measure on $I\times [0,1]$.

Note first that $S$ is the support of an ergodic absolutely continuous  $f_0$-invariant probability $\mu_0$ if and only if $h_t(S)$ is the support of an ergodic absolutely continuous  $f_t$-invariant probability $\mu_t$. Indeed by Boyarsky and  G\'{o}ra \cite{bg} the set $S$ is a $f_0$-invariant  finite union of intervals, so $h_t(S)$ is a $f_t$-invariant  finite union of intervals. That implies that $h_t(S)$ contain the support of  an ergodic absolutely continuous  $f_t$-invariant probability $\mu_t$. If we exchange  the roles of $0$ and $t$ we conclude that $h_t^{-1}(supp \ \mu_t)\subset S$ is $f_0$-invariant a finite union of intervals and the ergodicity of $\mu_0$ implies $h^{-1}(supp \ \mu_t)=S$.

Suppose that $(x_{1},t_{1}), (x_{2},t_{2})$, with $t_0\neq t_1$, belongs to $A\cap \gamma$, where $\gamma$ is a center leave. That means that there is $(x_0,0)$ such that $x_i=h_{t_i}(x_0)$, with $i=1,2$. Let $h= h_{t_1}\circ h_{t_0}^{-1}$ be the conjugacy  between $f_{t_0}$ and $f_{t_1}$. Then $h(x_0)=x_1$.  Note  that  the typically of $x_i$ imply 
$$\overline{\mathcal{O}}^+(x_i)= supp \ \mu_{x_i,t_i}=S_i$$
 so $h(supp \   \mu_{x_1,t_1})= supp \   \mu_{x_2,t_2}$ and
 $$\mu_{x_2,t_2}(h(B))=\mu_{x_1,t_1}(B)$$
 for every borelian set $B$.  Indeed the support $S_i$  is a finite  union of intervals and, since $\mu_{x_i,t_i} =\ \rho_i m$, where $\rho_i$ has a positive upper and lower bound on $S_i$ we conclude that $h$ is absolutely continuous with respect to the Lebesgue measure  in $S_1$ and in fact a bi-Lipschitz  function.  In particular
 $$\log |Df_{t_2}(h(x))| - log |Df_{t_1}(h(x))|= \log |Dh(f_{t_1}(x))|-  \log |Dh(x)|.$$
The left hand side is a piecewise Lipschitz function and $ \log |Dh(x)|\in L^\infty(S_1)$. G.R. and S. \cite[Theorem \ref{sum-lipsc}]{um} implies
 $$\sum_{j=0}^{M-1}\log |Df_{t_2}(f_{t_2}^j(h_{t_2}(q)))| = \sum_{j=0}^{M-1}\log |Df_{t_1}(f_{t_1}^j(h_1(q)))|$$
 for every $f_0$ periodic point $q$  in the support  of an  absolutely continuous ergodic probability of $f_0$. This is not possible.
\end{proof} 

The following is an immediate consequence of the results on deformation of piecewise expanding maps
\begin{proposition} We have
\begin{itemize}
\item[A.]  For every $f_0$ there are examples of deformations $f_t$ for which the holonomies between unstable leaves through the central  lamination are not absolutely continuous, and the central lamination is a Fubini's nightmare. 
\item[B] For every $f_0$ there are examples of deformations $f_t$ for which the holonomies between unstable leaves  through the central  lamination are quasisymmetric. However they are not absolutely continuous, and  the central lamination is a Fubini's nightmare. 
\end{itemize} 
\end{proposition} 

\begin{remark} Pathological invariant foliations, as foliations with atomic decomposition,  seem to be  ubiquitous   in partially hyperbolic dynamics, and have been intensively studied by many authors. See  Shub and  Wilkinson \cite{sw}, Ruelle  and Wilkinson \cite{rw}, Hirayama and  Pesin \cite{hp}, Homburg \cite{h}, Gogolev and  Tahzibi \cite{gt} and Avila, Viana and Wilkinson \cite{avw}. 

Examples similar to Katok's example  as those in this section are  quite special cases. However statistical properties of the distribution of the central direction similar to Proposition \ref{clt} does not seem to appear in the previous literature. One may wonder if similar statistical properties hold  for  more general classes of partially hyperbolic maps, and if they can help to understand  their dynamics.
\end{remark}

\section{Pressure pseudo-metric on the topological class}

Once we know that the topological class $\mathcal{T}(f_0)$ of a piecewise expanding map $f_0$ is  a Banach manifold, one may ask if there is an interesting, {\it dynamically defined}  riemannian (pseudo-)metric on  $\mathcal{T}(f_0)$. The work of McMullen \cite{mct} on the characterisation via thermodynamical formalism of the Weil–Petersson metric on the Teichm\"uller  space (and its generalisations for Blaschke products) suggest that a "nice" dynamically-defined pseudo-metric would be the {\it pressure pseudo-metric} 
$$<v_1,v_2>_{E^h(f_0)}=\sigma(\frac{Dv_1+D^2f\cdot\alpha_1}{Df},\frac{Dv_2+D^2f\cdot\alpha_2}{Df}).$$
where $\sigma$ is the hermitian form
$$\sigma(\phi_1,\phi_2)=  \lim_{N \rightarrow \infty}  \int \Big( \frac{\sum_{i=0}^{N-1} \phi_1\circ f^i }{\sqrt{N}}   \Big)\Big( \frac{\sum_{i=0}^{N-1} \overline{\phi}_2 \circ f^i }{\sqrt{N}}   \Big)   \ dm,$$
that due G.R. and S. \cite{um} is well defined for every pair $(\phi_1,\phi_2)\in \mathcal{B}^\beta(C)$ such that 
$$\int \phi_i \Phi_1(\gamma) \ dm =0$$
for every $\gamma\in BV$, $i=1,2$, and $v_i=\alpha_i\circ f -Df\cdot \alpha_i$, where $\alpha_i$ are Log-Lipschitz. Note that $m$ does not need to be $f$-invariant.  One must compare this with 
Giulietti,  Kloeckner,  Lopes,  and Marcon \cite{giu}, a study  of thermodynamical formalism in a geometric framework.   See also Pollicott  and Sharp \cite{ps} and Bridgeman, Canary and  Sambarino \cite{brig} and the Weil-Petersson metric  in the infinite-dimensional Teichm\"uler space in  Takhtajan and  Teo \cite{tt}.

There are many interesting questions one can ask on this pseudo-metric. We give a result  that follows immediately from our results on Birkhoff sums as distributions and deformations.

\begin{proposition} Let $w\in E^h(f)$ and define
$$\Theta(v)= <v,w>_{E^h(f)}.$$
Let $$\phi= \frac{Dw+D^2f\cdot\alpha}{Df}.$$
The  following statements are equivalent
\begin{itemize}
\item[A.]  $\Theta$ is a signed measure.
\item[B.] $\phi=\psi\circ f -\psi$, where $\psi\in L^2(m)$ and $\psi\in L^\infty(S_\ell)$ for every $\ell\leq E$.
\item[C.]  $\Theta=0$.
\end{itemize}
Moreover $A.-C.$ implies
\begin{itemize}
\item[D.]  We have that 
\begin{equation}\label{inficoddd} \sum_{j=0}^{M-1} \phi(f^j(q))=0\end{equation} 
holds for every $M$ and $q\in \hat{S}_\ell$, with $\ell\leq E$,  such that $f^M(q)=q$.
\end{itemize} 
Furthermore if $f$ is markovian, $p(f)=1$ and it has an absolutely continuous ergodic invariant probability whose support is $I$ then $D.$ is equivalent to $A.-C.$
\end{proposition} 
\begin{proof}  This follows from G.R. and S. \cite[Theorem \ref{sum-xcxc}]{um}. 
\end{proof}

\bibliographystyle{abbrv}
\bibliography{bibliografia}

\end{document}